\documentclass[12pt]{amsart}
\usepackage[utf8]{inputenc}
\usepackage[T1]{fontenc}

\usepackage{CormorantGaramond} 

\usepackage[margin=1in]{geometry}

\RequirePackage{color}
\usepackage{xcolor}

\makeatletter

\usepackage{amsmath,amsfonts,amssymb,amsthm,epsfig,epstopdf,titling,url,array,mathtools,faktor,paralist, xfrac, mathrsfs,pifont,yfonts,mathabx,leftindex,adjustbox,stmaryrd, mathdots}
\usepackage[new]{old-arrows}
\usepackage{ytableau}
\usepackage[super]{nth}
\usepackage[foot]{amsaddr}

\theoremstyle{plain}
\newtheorem{thm}{Theorem}[section]
\newtheorem*{unnumthm}{Theorem}
\newtheorem{thmx}{Theorem}

\newtheorem{lem}[thm]{Lemma}
\newtheorem{prop}[thm]{Proposition}
\newtheorem{propx}[thmx]{Proposition}

\newtheorem{cor}[thm]{Corollary}

\theoremstyle{definition}
\newtheorem{defn}{Definition}[section]
\newtheorem*{unnumdefn}{Definition}

\newtheorem{ex}{Example}[section]
\newtheorem{rmk}[thm]{Remark}
\newtheorem{ques}[thm]{Question}

\DeclareMathOperator{\add}{\mathsf{add}}
\DeclareMathOperator{\F}{\mathcal{F}}
\DeclareMathOperator{\tr}{Tr}
\DeclareMathOperator{\op}{op}

\DeclareMathOperator{\Hom}{Hom}
\DeclareMathOperator{\Ext}{Ext}
\DeclareMathOperator{\GL}{GL}

\DeclareMathOperator{\tilt}{-\,\mathsf{tilt}}
\DeclareMathOperator{\ind}{\operatorname{-\,ind}}
\DeclareMathOperator{\MOD}{\operatorname{-\,MOD}}

\renewcommand\mod{{\operatorname{\,-\,mod}}}

\begingroup\lccode`~=`& \lowercase{\endgroup
\newcommand{\spmat}[1]{%
  \left[
  \let~=&
  \begin{smallmatrix}#1\end{smallmatrix}
  \right]
}}

\definecolor{bettergreen}{RGB}{0,125,0}
\definecolor{alizarin}{rgb}{0.85, 0.15, 0.26}
\definecolor{azure}{rgb}{0.0, 0.5, 1.0}
\definecolor{upmaroon}{rgb}{0.70, 0.07, 0.07}

\definecolor{DavidGreen}{RGB}{29,162,55}

\usepackage{tikz}
\usepackage{tikz-cd}
\usetikzlibrary{arrows,positioning,scopes,trees}

 \usepackage{hyperref}
 \hypersetup
 {
     colorlinks=true,
     linktocpage=true,
     linkcolor=upmaroon,
     filecolor=magenta,
     citecolor=bettergreen,
     urlcolor=azure,
     pdftitle={Galois Coverings, \texorpdfstring{$\tau$}{tau}-Rigidity and Mutations}
 }

\makeatother

\title{Galois coverings, $\tau$-rigidity and mutations}
\author{Charles Paquette$^{1,2}$, Deepanshu Prasad$^2$, and David Wehlau$^{1,2}$}
\address{$^1$Department~of Mathematics and Computer Science, Royal Military college of canada\\
$^2$Department of mathematics and Statistics, Queen's university, Kingston on,Canada}
\email{charles.paquette.math@gmail.com, deepanshu.prasad@gmail.com, wehlau@rmc.ca}

\begin{document}

\begin{abstract}
For an algebraically closed field $\mathbb{K}$, we consider a Galois $G$-covering $\mathcal{B} \to \mathcal{A}$ between locally bounded $\mathbb{K}$-categories given by bound quivers, where $G$ is torsion-free and acts freely on the objects of $\mathcal{B}$. We define the notion of $(G,\tau_\mathcal{B})$-rigid subcategory and of support $(G,\tau_\mathcal{B})$-tilting pairs over $\mathcal{B}$-$\rm mod$. These are the analogues of the similar concepts in the context of a finite dimensional algebra, where we additionally require that the subcategory be $G$-equivariant. When $\mathcal{A}$ is a finite dimensional algebra, we show that the corresponding push-down functor $\mathcal{F}_\lambda: \mathcal{B}$-$\rm mod$ $\to \mathcal{A}$-$\rm mod$ sends $(G,\tau_\mathcal{B})$-rigid subcategories (respectively support $(G,\tau_\mathcal{B})$-tilting pairs) to $\tau_\mathcal{A}$-rigid modules (respectively support $\tau_\mathcal{A}$-tilting pairs). We further show that there is a notion of mutation for support $(G,\tau_\mathcal{B})$-tilting pairs over $\mathcal{B}$-$\rm mod$. Mutations of support $\tau_\mathcal{A}$-tilting pairs and of support $(G,\tau_\mathcal{B})$-tilting pairs commute with the push-down functor. We derive some consequences of this, and in particular, we derive a $\tau$-tilting analogue of the result of P. Gabriel that locally representation-finiteness is preserved under coverings. Finally, we prove that when the Galois group $G$ is finitely generated free, any rigid $\mathcal{A}$-module (and in particular $\tau_\mathcal{A}$-rigid $\mathcal{A}$-modules) lies in the essential image of the push-down functor.

\end{abstract}

\maketitle

\section{Introduction}
\phantomsection\label{intro}

Let $\mathbb{K}$ be an algebraically closed field. We let $\mathcal{A}$ denote a finite dimensional $\mathbb{K}$-algebra, which we assume to be connected, without no loss of generality.
Our focus is on the representation theory of $\mathcal{A}$,
specifically the module category $\mathcal{A}\mod$ of finitely generated left $\mathcal{A}$-modules. An important class of modules is the class of tilting modules, which traces back to the work of Bernstein, Gelfand and Ponomarev \cite{Bernstein1973-vt}, and later well studied in the context of representation theory of finite dimensional algebras; see \cite{Happel1982-tf}, \cite{MR0607151}, and \cite{MR0654701}. These modules serve as a generalization of progenerators from classical Morita theory and relate closely to Bernstein-Gelfand-Ponomarev (BGP) reflections in quiver representations, as discussed in \cite{Bernstein1973-vt}. Notably, BGP-reflections are a special class of tilting modules called APR-tilting modules, see \cite{MR0530043}.

One of the important properties of tilting modules
is that any almost-complete tilting module has either one or two complements, making possible a partially defined notion of mutation. To make mutation always possible, a bigger class of modules having the more regular property that almost-complete ones always have two complements is desirable. This is well accomplished by introducing support $\tau$-tilting modules or pairs, see \cite{AIR}. 

This latter property is particularly significant in the context of categorification of cluster algebras. In the realm of cluster categories, as explored in \cite{BUAN2006572}, and more generally in 2-Calabi–Yau triangulated categories, as examined in \cite{Iyama2008-fg}, there exists a class of objects known as cluster-tilting objects, which possess the desired property of ensuring the existence of two complements.

On the other hand, coverings techniques in representation theory were introduced and developed to understand many aspects of representation theory, including representation-finite algebras and their indecomposable representations. In this theory, one of the important results can be summarized as follows, see \cite{Bongartz-Gabriel}, \cite{MartPenaInven}: Let $Q$ be a locally finite quiver, $I$ an admissible ideal in the path-category $\mathbb{K}Q$ of $Q$, $\mathcal{R}$ a locally bounded quotient category $\mathbb{K}Q/I$, $\mathcal{R}\mod$ the category of finite dimensional left $\mathcal{R}$-modules (or representations of $\mathcal{R}$) and $G$ a group of $\mathbb{K}$-linear automorphism of $\mathcal{R}$ acting freely on the objects of $\mathcal{R}$. Moreover, let $\mathcal{F}:\mathcal{R} \rightarrow \mathcal{R}/G \eqqcolon \mathcal{A}$ be the functor which assigns to each object $x$ of $\mathcal{R}$ its $G$-orbit $G.x$, and $\mathcal{F}_{\lambda}:\mathcal{R}\mod  \rightarrow (\mathcal{R}/G)\mod$ the push-down functor associated with $\mathcal{F}$. Then $\mathcal{R}$ is locally representation-finite if and only if $\mathcal{R}/G$ is. In this case, $\mathcal{F}_{\lambda}$ induces a bijection between the $G$-orbits of isomorphism classes of indecomposable finite dimensional $\mathcal{R}$-modules and the isomorphism classes of indecomposable finite dimensional $\mathcal{A}$-modules.

Covering techniques have also been used to study several other properties of the finite dimensional algebras. These include the relation between the covering of the Auslander-Reiten quiver and the ordinary quiver (\cite{MartPenaInven}), conditions on the algebra so that the fundamental group of the bound quiver is a free group (\cite{Caston-de}), the representation finiteness of standard algebras (\cite{Bongartz1984}), and many more. 

Given that both $\tau$-tilting theory and covering techniques have proven valuable in the study of representation theory for finite-dimensional algebras, we aim to develop a theory of mutation in the context of Galois coverings.

Suppose $\mathcal{A}$ is given by a finite, connected bound quiver $(Q,I)$. Let $p : (\Gamma,J)\rightarrow (Q,I)$ be a Galois covering with $G$ as a group of automorphisms of $(\Gamma,J)$. We further assume that $G$ is torsion-free. Let $\mathcal{B}\coloneqq\mathbb{K}\Gamma/J$. Note that both $\mathcal{A}$ and $\mathcal{B}$ are locally bounded $\mathbb{K}$-linear categories. By \cite[Theorem, Section 3.6, page 89]{GabUnivCov}, we know that the Auslander-Reiten translate exists for $\mathcal{B}\mod$. We denote by $\tau_{\mathcal{B}}$ and $\tau_{\mathcal{A}}$ the Auslander-Reiten translates for $\mathcal{B}$ and $\mathcal{A}$, respectively.

For a locally bounded $\mathbb{K}$-linear category $\mathcal{R}$, we denote by $\mathcal{R}\MOD$ (resp., $\mathcal{R}\mod$) the category of all left $\mathcal{R}$-modules (resp., finite dimensional left $\mathcal{R}$-modules). By a subcategory $\mathcal{C}\subseteq \mathcal{R}\mod$ we mean a subcategory of $\mathcal{R}\mod$ that is full, skeletal, and closed under finite direct sums and direct summands. We will use $\mathcal{R}\ind$ to denote the subcategory of $\mathcal{R}\mod$ formed by indecomposable objects. Denote by $\mathsf{proj}\,\mathcal{R}$ the subcategory of $\mathcal{R}\mod$ formed by projective $\mathcal{R}$-modules.

Let $M\in \mathcal{B}\mod$. The group $G$ acts on $\mathcal{B}\mod$ and for $g \in G$, we denote by $M^{g}$ the $g$-translate of $M$, and by $\mathcal{O}(M)\coloneqq\lbrace M^{g}\mid g\in G\rbrace$ the $G$-orbit of $M$. Let $\mathcal{C}\subseteq\mathcal{B}\mod$ be a subcategory. We will say that $\mathcal{C}$ is \emph{$G$-equivariant} if, for every $\mathcal{B}$-module $M\in\mathcal{C}$, the module $M^g\in\mathcal{C}$ for every $g \in G$. For a $G$-equivariant subcategory $\mathcal{C}$ we will use $\lvert\mathcal{C}\rvert_{G}$ to denote the number of $G$-orbits of distinct indecomposable objects in $\mathcal{C}$. For subcategories $\mathcal{C}_1,\mathcal{C}_2\subseteq\mathcal{B}\mod$, we denote by $\mathsf{add}(\mathcal{C}_1,\mathcal{C}_2)$ the smallest (additive) subcategory of $\mathcal{B}\mod$ containing both $\mathcal{C}_1$ and $\mathcal{C}_2$, and by $\mathsf{add}_{G}(\mathcal{C}_1,\mathcal{C}_2)$ the smallest $G$-equivariant (additive) subcategory of $\mathcal{B}\mod$ containing both $\mathcal{C}_1$ and $\mathcal{C}_2$.

We define the notion of support $(G,\tau_{\mathcal{B}})$-tilting pair in the following way:

\begin{unnumdefn}[see Definition \ref{(G,tau)-tilting} and \ref{supp(G,tau)-tilting}]
    Let $(\mathcal{T},\mathcal{P})$ be a pair where $\mathcal{T}\subseteq\mathcal{B}\mod$ and $\mathcal{P}\subseteq\mathsf{proj}\,\mathcal{B}$ are  subcategories
    \begin{enumerate}
        \item We say $\mathcal{T}$ is $\emph{$\tau_{\mathcal{B}}\,$-rigid}$ if $\Hom_{\mathcal{B}}(X,\tau_{\mathcal{B}}Y)=0$, for every $X,Y\in\mathcal{T}$.
        \item The pair $(\mathcal{T},\mathcal{P})$ is said to be a $\emph{$(G,\tau_{\mathcal{B}})$-rigid pair}$ if the following conditions are satisfied:
        \begin{enumerate}
            \item $\mathcal{T}$ is $\tau_{\mathcal{B}}\,$-rigid
            \item Both $\mathcal{T}$ and $\mathcal{P}$ are $G$-equivariant
            \item $\Hom_{\mathcal{B}}(\mathcal{P},\mathcal{T})=0$
        \end{enumerate}
        \item The pair $(\mathcal{T},\mathcal{P})$ is said to be a $\emph{support $(G,\tau_{\mathcal{B}})$-tilting pair}$ (respectively, $\emph{almost complete support}$ $\emph{$(G,\tau_{\mathcal{B}})$-tilting pair}$) if $(\mathcal{T},\mathcal{P})$ is a $(G,\tau_{\mathcal{B}})$-rigid pair and $\lvert \mathcal{T}\rvert_{G}+\lvert \mathcal{P}\rvert_{G}=\lvert \mathcal{A}\rvert$ (resp., $\lvert \mathcal{T}\rvert_{G}+\lvert \mathcal{P}\rvert_{G}=\lvert \mathcal{A}\rvert-1$).
    \end{enumerate}
\end{unnumdefn}

In contrast to the classical mutation process, where we usually mutate a single indecomposable direct summand of a support \(\tau_{\mathcal{A}}\)-tilting pair, we will mutate an entire \(G\)-orbit of an indecomposable object within a support \((G, \tau_{\mathcal{B}})\)-tilting pair. More precisely, we make the following definition.

\begin{unnumdefn}[see Definition \ref{equivsubcatpair} and \ref{(G,tau)mut}]
    Let $(\mathcal{T}_1,\mathcal{P}_1),(\mathcal{T}_2,\mathcal{P}_2)$ be two non-equivalent support $(G,\tau_{\mathcal{B}})$-tilting pairs. They are said to be $\emph{$(G,\tau_{\mathcal{B}})$-mutations}$ of each other if there exists an almost complete support $(G,\tau_{\mathcal{B}})$-tilting pair $(\mathcal{T},\mathcal{Q})$ such that $(\mathcal{T},\mathcal{Q})\prec (\mathcal{T}_1,\mathcal{P}_1)$ and $(\mathcal{T},\mathcal{Q})\prec (\mathcal{T}_2,\mathcal{P}_2)$.
    
    If $X\in\mathcal{B}\ind$ satisfies either $\mathcal{T}_1=\mathsf{add}_{G}(\mathcal{T},X)$ and $\mathcal{P}_1 = \mathcal{Q}$; or $\mathcal{T}_1 = \mathcal{T}$ and $\mathcal{P}_1=\mathsf{add}_{G}(\mathcal{Q},X)$, then we write $\mu_{\mathcal{O}(X)}(\mathcal{T}_1)=\mathcal{T}_2$ to denote the mutation of $(\mathcal{T}_1,\mathcal{P}_1)$ at the $G$-orbit $\mathcal{O}(X)$.
\end{unnumdefn}

Let $s(G,\tau_{\mathcal{B}})\tilt$ denote the set of all support $(G,\tau_{\mathcal{B}})$-tilting pairs (up to equivalence), and $s\tau_{\mathcal{A}}\tilt$ denote the set of all basic support $\tau_{\mathcal{A}}$-tilting pairs (up to isomorphism). We note that these sets have the natural structure of a poset. To every support $(G, \tau_\mathcal{B})$-rigid pair $(\mathcal{T}, \mathcal{P})$, we can associate a torsion class ${\rm Fac}(\mathcal{T})$ and torsion classes over $\mathcal{B} \mod$ form a poset by inclusion. We have a similar observation for support $\tau_\mathcal{A}$-tilting modules in $\mathcal{A} \mod$. We will interpret an element of $s\tau_{\mathcal{A}}\tilt$ as a pair of subcategories of $\mathcal{A}\mod$, see Remark \ref{pair of subcategories instead of pair of mod}.

Using the above terminologies, we get the following results. Our first main result tells us that the push-down functor is compatible with both the mutation of $G$-orbits and mutation of (support) $\tau_{\mathcal{A}}$-tilting modules. 

\begin{thmx}
    We have the following commutative diagram
    \begin{equation*}
        \begin{tikzcd}
            s(G,\tau_{\mathcal{B}})\tilt \arrow[r,"{\mathcal{F}_{\lambda}}"] \arrow[d,"{\mu}"] & s\tau_{\mathcal{A}}\tilt \arrow[d,"{\mu}"] \\
            s(G,\tau_{\mathcal{B}})\tilt \arrow[r,"{\mathcal{F}_{\lambda}}"] & s\tau_{\mathcal{A}}\tilt
        \end{tikzcd}
    \end{equation*}
\end{thmx}

The above theorem yields a similar result to the result of P. Gabriel, see \cite{GabUnivCov} Lemma 3.3, Theorem 3.6, also see \cite{MARTINEZVILLA1983277}, in terms of $\tau_{\mathcal{A}}$-tilting theory, as follows.

\begin{thmx}
    The Galois covering $\mathcal{B}$ is locally $(G,\tau_{\mathcal{B}})$-tilting finite (see Definition \ref{loc(G,tau)tiltingfin}) if and only if $\mathcal{A}$ is $\tau_{\mathcal{A}}$-tilting finite.
\end{thmx}

Another line of research has been to understand the orbit category, for example, see \cite{AmiOpper}, and the image of the push-down functor, see \cite{EdGreen}, \cite{GreenVillaYoshino}. It has also been used in the classification of indecomposable modules over interesting classes of tame algebras, see \cite{Dowbor1987}. Moreover, a result of W.~Crawley-Boevey states that if $\mathrm{char}(\mathbb{K})=0$ and $G=\mathbb{Z}$, then any indecomposable $\mathcal{A}$-module $M$ with $\Ext^{1}_{\mathcal{A}}(M,M)=0$ is in the essential image of the push-down functor $\F_{\lambda}$.

We call an $\mathcal{A}$-module $V$ a \emph{module of first kind with respect to $\F_{\lambda}$} if and only if for every direct summand $V'$ of $V$ there exists a $W'\in \mathcal{B}\MOD$ such that $\F_{\lambda} W'=V'$. We denote by $\mathcal{A}\mod_1$ the full subcategory of $\mathcal{A}\mod$ of modules of first kind with respect to $\F_{\lambda}$. Below, we let $Q(s\tau_{\mathcal{A}}\tilt)$ denote the Hasse quiver of $s\tau_{\mathcal{A}}\tilt$, where arrows represent covering relations, equivalently mutations, of the poset $s\tau_{\mathcal{A}}\tilt$. We show that

\begin{propx}
    If $(T_1,P_1),(T_2,P_2)\in s\tau_{\mathcal{A}}\tilt$ are in the same connected component of $Q(s\tau_{\mathcal{A}}\tilt)$, then
    \[
        (T_1,P_1)\in \mathcal{A}\mod_1 \;\text{if and only if}\;\,(T_2,P_2)\in \mathcal{A}\mod_1
    \]
    In particular, if $(T,P)\in s\tau_{\mathcal{A}}\tilt$ is such that it belongs to the connected component of $Q(s\tau_{\mathcal{A}}\tilt)$ with $(\mathcal{A},0)$ or $(0,\mathcal{A})$ as one of its vertices, then $(T,P)\in \mathcal{A}\mod_1$
\end{propx}

The above proposition can be seen as a generalisation of \cite[Proposition 3.1]{LEMEUR}. As a consequence of the above proposition we see that if $Q(s\tau_{\mathcal{A}}\tilt)$ is connected, then every $\tau_{\mathcal{A}}$-rigid $\mathcal{A}$-module belongs to $\mathcal{A}\mod_1$. In particular, if $\mathcal{A}$ is a hereditary algebra or if $\mathcal{A}$ is a $\tau_{\mathcal{A}}$-tilting finite algebra, then every $\tau_{\mathcal{A}}$-rigid $\mathcal{A}$-module lies in $\mathcal{A}\mod_1$.

Recall that a $\tau_{\mathcal{A}}$-rigid $\mathcal
A$-module $M$ has an open orbit under the natural $\GL(\dim M,\mathbb{K})$-action in the variety of $(\dim M)$-dimensional $\mathcal{A}$-modules. Hence, we can ask the following question: 

\begin{ques}
\phantomsection\label{isinimg?}
    Let $\mathcal{A}$ be a finite dimensional, connected, and basic $\mathbb{K}$-algebra. If $M\in\mathcal{
    A}\mod$ is a module such that $M$ has an open orbit under the natural $\GL(\dim M,\mathbb{K})$-action in the variety of $(\dim M)$-dimensional $\mathcal{A}$-modules, then is $M\in\mathcal{A}\mod_1$ with respect to $\F_{\lambda}$?
\end{ques}

As a consequence of a result of E. L. Green, see \cite[Theorem 3.2]{EdGreen} or \cite[Theorem 4]{GreenVillaYoshino}, we get that the modules in $\mathcal{A}\mod_1$ and $G$-gradable $\mathcal{A}$-modules are closely related. Also, as seen in \cite[Theorem 3.1 and Theorem 5.4]{AmiOpper}, when $\mathcal{A}$ is $\mathbb{Z}$-graded, then $\mathbb{Z}$-gradable $\mathcal{A}$-modules play an important role.
This leads to the next main result of this article.

\begin{thmx}
\phantomsection\label{thmD}
    Assume $G$ is a finitely-generated free group. If $M\in \mathcal{A}\mod$ is a rigid $\mathcal{A}$-module, then $M\in \mathcal{A}\mod_1$, and hence, is $G$-gradable.
\end{thmx}

\section{Preliminaries}
\phantomsection

Let $\mathbb{K}$ be an algebraically closed field, and let $Q$ be a quiver (not necessarily finite). We will denote by $Q_0$ the set of its $\emph{vertices}$, and by $Q_1$ the set of its $\emph{arrows}$. We have two functions $h,t:Q_1 \rightarrow Q_0$ defined in the following way: let $a \in Q_1$ be an arrow, then $h(a)$ and $t(a)$ are called the $\emph{head}$ and the $\emph{tail}$ of $a$, respectively. If $Q_0$ and $Q_1$ are finite sets, then we say $Q$ is a $\emph{finite quiver}$.

A $\emph{path}$ $p = a_\ell a _{\ell-1} \dotsb a _1$ in $Q$ is defined as a sequence  of arrows in $Q _1$ such that $ta_{i+1} = ha_ i$ for $i = 1,2,\dots,\ell - 1$. We call the integer $\ell$ the \emph{length} of $p$ and define $h(p) = hp = ha_\ell$ and $t(p) = tp = ta_ 1$. For every $x \in Q_ 0$, we introduce a $\emph{trivial path}$ $e_x$ of length 0, which satisifies $h(e_x) = t(e_x) = x$. The $\emph{path category}$ $\mathbb{K}Q$ is defined as follows: its object set is $Q_0$, and, given $x, y \in Q_0$, the morphism set $\mathbb{K}Q(x, y)$ is the $\mathbb{K}$-vector space having as basis the set of paths from $x$ to $y$. The composition $w_1 \circ w_2$ of two paths $w_1 \in \mathbb{K}Q(y,z)$ and $w_2 \in \mathbb{K}Q(x,y)$ is given by the concatenation of paths: $w_1 \circ w_2 = w_1w_2$. Composition of morphisms is induced from this using bilinearity.

Let $I \subseteq \mathbb{K}Q$ be a two sided ideal. We can consider $\mathcal{A}\coloneqq\mathbb{K}Q/I$ as a $\mathbb{K}$-linear category with object class $Q_0$ and the set of morphisms from $x$ to $y$ is the quotient space of $\mathbb{K}Q(x,y)$ by all paths in $I(x,y)$, which we denote by $\mathcal{A}(x,y)$. The $\mathbb{K}$-linear category $\mathcal{A}$ is said to be $\emph{locally bounded}$ if the following three conditions are satisfied:
\begin{enumerate}
    \item for each object $x\in \mathcal{A}$, the endomorphism algebra $\mathcal{A}(x,x)$ is local.
    \item distinct objects in $\mathcal{A}$ are non-isomorphic
    \item  for each object $x\in \mathcal{A}$, we have $\dim_{\mathbb{K}}\left(\oplus_{y\in \mathcal{A}}\mathcal{A}(x,y)\right)<\infty$ and $\dim_{\mathbb{K}}\left(\oplus_{y\in \mathcal{A}}\mathcal{A}(y,x)\right)<\infty$
\end{enumerate}

\medskip

The quiver $Q$ is $\emph{locally finite}$ if the sets $h^{-1}(x)$ and $t^{-1}(x)$ are finite for every vertex $x \in Q_0$. Let $\mathfrak{m}$ be the ideal of $\mathbb{K}Q$ generated by the arrows of $Q$, that is, the ideal whose value $\mathfrak{m}(x, y)$ at a pair of vertices $(x, y)$ is the $\mathbb{K}$-vector space generated by the paths of positive length from $x$ to $y$. A two sided ideal $I \subseteq \mathbb{K}Q$ is said to be $\emph{admissible}$ if for every $x, y \in Q_0$ one has $I(x,y) \subseteq \mathfrak{m}^2(x,y)$ and for every $x \in Q_0$ there exists a natural number $n_x \geq 2$ such that $I$ contains every path of length greater than $n_x$ having $x$ as head or tail. In this situation, the pair $(Q,I)$ will be called a $\emph{bound quiver}$. When an admissible ideal is clearly understood from the context, a $\emph{relation}$ in $\mathbb{K}Q$ is an element of that admissible ideal.

Note that when $I$ is an admissible ideal, the first two conditions above are automatically satisfied. Additionally, if $Q$ is locally finite, then the third condition holds.

\subsection{Homotopy relations for bound quivers}

Let $(Q, I)$ be a bound quiver, and let $x,y$ be two vertices of $Q$. Consider the relation $\rho=\sum_{i=1}^{r} \lambda_i w_i$, where $\lambda_i \in\mathbb{K}^{*}$ and the $w_i$ are pairwise distinct paths of length at least two from $x$ to $y$. We say $\rho$ is a $\emph{minimal relation}$ if:

\begin{enumerate}
    \item $r\geq2$, and
    \item for every non-empty proper subset $S$ of $\lbrace1,\dotsc,r\rbrace$ we have $\sum_{i\in S} \lambda_i w_i \notin I$.
\end{enumerate}

Given an arrow $\alpha : x \rightarrow y$ in $Q$, we denote by $\alpha^{-1}$ its formal inverse of head $h(\alpha^{-1})=x$
and tail $t(\alpha^{-1})=y$. A $\emph{walk}$ $w$ from a vertex $x$ to a vertex $y$ is a sequence $w = \alpha_{n}\dotsb\alpha_{2}\alpha_{1}$ of arrows and formal inverses of arrows such that $t(\alpha_{1})=x$, $h(\alpha_{n})=y$, and $t(\alpha_{i})=h(\alpha_{i-1})$ for every $1<i\leq n$. If $w$ and $w'$ are walks in $Q$ such that $t(w')=h(w)$, one can consider the walk $w'w$. This yields a product on the set of walks, which, however, is not everywhere defined. We define the $\emph{homotopy relation}$ on the set of walks on $Q$ as the smallest equivalence relation $\sim$ satisfying:

\begin{enumerate}
    \item For each arrow $\alpha: x \rightarrow y$, one has $\alpha^{-1}\alpha\sim e_x$ and $\alpha\alpha^{-1}\sim e_{y}$.
    \item For each minimal relation $\sum_{i=1}^{r} \lambda_i w_i$, one has $w_i \sim w_j$, for all $i, j \in \lbrace1,\dotsc,r\rbrace$.
    \item If $u, v, w$ and $w'$ are walks, and $u \sim v$, then $wuw'\sim  wvw'$, whenever these
    compositions are defined.
\end{enumerate}

We denote by $[w]$ the homotopy class of a walk $w$.

\subsection{Galois coverings of bound quivers}

Let $(Q,I)$ be a bound quiver with $Q$ connected, and fix a vertex $x_0 \in Q_0$. Let $\pi_1(Q,x_0)$ be the fundamental group of the underlying graph of $Q$ at the vertex $x_0$. We have that $\pi_{1}(Q,x_0)$ is a free group, and if $Q$ is finite, then $\pi_{1}(Q,x_0)$ is isomorphic to the free group on $\chi(Q) = \vert Q_1\rvert - \lvert Q_0\rvert + 1$ generators, see \cite{Rotman2013-io}. Moreover, let $N(Q,I,x_0)$ be the normal subgroup of $\pi_1(Q,x_0)$ generated by all the elements of the form $wuv^{-1}w^{-1}$, where $w$ is a walk from $x_0$ to, say $x$, and $u, v$ are two homotopic paths starting and ending at $x$.

\begin{defn}
    Let $(Q,I)$ be a connected bound quiver. The fundamental group is then defined to be $\pi_1(Q,I) = \pi_1(Q,x_0)/N(Q,I,x_0)$.
\end{defn}

In order to define the notion of Galois covering, we first need to define morphisms of bound quivers, as follows.

\begin{defn}
    Let $\Gamma,Q$ be two quivers.
    \begin{enumerate}
        \item A $\emph{quiver morphism}$, written as $f:\Gamma\rightarrow Q$, is a pair of maps $f_i:\Gamma_i\rightarrow Q_{i}$, for $i\in\lbrace0,1\rbrace$, such that every arrow $\alpha:x\rightarrow y$ in $\Gamma$ induces an arrow $f_1(\alpha):f_0(x)\rightarrow f_0(y)$ in $Q$.

        Given such a morphism, we can extend it naturally to paths in $\Gamma$: if $p = a_\ell a _{\ell-1} \dotsb a _1$ is a path of $\Gamma$ starting at $x$ and ending at $y$, then $f(p)\coloneqq f_1(a_\ell) f_1(a _{\ell-1}) \dotsb f_1(a _1)$ is a path in $Q$ starting at $f_0(x)$ and ending at $f_0(y)$. Hence, it induces in a natural way a $\mathbb{K}$-linear map $f :\mathbb{K}\Gamma\rightarrow\mathbb{K}Q$.
        \item A $\emph{bound quiver morphism}$, written as $f:(\Gamma,J)\rightarrow(Q,I)$, is given by a quiver morphism $f:\Gamma\rightarrow Q$ such that $f(J)\subseteq I$. 
    \end{enumerate}
\end{defn}

We can think of $f$ as being a functor from $\mathbb{K}\Gamma/J$ to $\mathbb{K}Q/I$ and we will make this identification in the sequel. Given a bound quiver $(Q,I)$ and a vertex $x\in Q_{0}$, we denote by $x^{\rightarrow}$ (resp. $^{\rightarrow}x$) the set of outgoing arrows (resp. incoming arrows) at vertex $x$. Recall the following definition from \cite{DELAPENA1986129}.

\begin{defn}
    A $\emph{Galois $G$-covering}$ (or Galois covering) of a bound quiver $(Q,I)$ is a bound quiver morphism $p : (\Gamma,J)\rightarrow (Q,I)$ together with a group $G$ of automorphisms of $(\Gamma,J)$ such that:

    \begin{enumerate}
        \item $G$ acts freely on $\Gamma_0$,
        \item $pg=p$ for every $g\in G$ and, for $x,y\in\Gamma_{0}$, $p(x)=p(y)$ if and only if there exists $g\in G$ such that $y = g(x)$.
        \item For every $x \in Q_0$ and $\hat{x} \in p^{-1}(x)$ the map $p$ induces bijections $\hat{x}^{\rightarrow} \rightarrow x^{\rightarrow}$ and $^{\rightarrow}\hat{x} \rightarrow ^{\rightarrow}x$.
        \item $I$ is the ideal generated by the elements of the form $p(\rho)$, with $\rho\in J$. 
    \end{enumerate}
    The group $G$ is called the \emph{Galois group} of the covering.
\end{defn}

The $\emph{universal Galois covering}$, see \cite{MARTINEZVILLA1983277}, $\pi : (\widetilde{Q},\widetilde{I})\rightarrow (Q,I)$ given by the group $G =\pi_1(Q,I)$ is the one of particular interest among all the Galois coverings of $(Q,I)$. More precisely, the universal cover $(\widetilde{Q},\widetilde{I})$ at some fixed vertex $x_0 \in Q_0$ is defined in the following way: The vertices of $\Gamma$ are the homotopy classes $[w]$ of walks of $Q$ starting at $x_0$, and there is an arrow from $[w]$ to $[w']$ whenever $w' = \alpha w$, with $\alpha\in Q_1$. This provides a map of quivers $\pi : \widetilde{Q} \rightarrow Q$ defined by $\pi([w]) = x$, where $x=h(w)$. Finally, the ideal $\widetilde{I}$ is defined to be generated by the inverse images under $\pi$ of the generators of $I$.

If there is another Galois covering $p : (\Gamma, J )\rightarrow(Q,I)$
given by a group $H$, then there exists a Galois covering $\pi^{*}:(\widetilde{Q},\widetilde{I})\rightarrow(\Gamma,J)$ such that $\pi = p\pi^{*}$.

\begin{equation}
\phantomsection\label{UnivPropOfGaloisCover}
    \begin{tikzcd}
	   {(\widetilde{Q},\widetilde{I})} && {(\Gamma,J)} \\
	   \\
	   {(Q,I)}
	   \arrow["{\pi^{*}}", dashed, from=1-1, to=1-3]
	   \arrow["\pi"', from=1-1, to=3-1]
	   \arrow["{p}", from=1-3, to=3-1]
    \end{tikzcd}
\end{equation}

Moreover, in this situation there is a normal subgroup $N$ of $\pi_1(Q, I)$ such that $\pi_1(\Gamma, J) \cong N$ and $\pi_1(Q, I)/N \cong H$, see \cite{BSMF_1983__111__21_0}, \cite{EdGreen}, \cite{MARTINEZVILLA1983277}.

\subsection{Pull-up and push-down functors}

In this section we define two special functors associated with the Galois covering of bound quivers, and look at some of their properties.

\begin{defn}
    Let $\mathcal{F}:\mathcal{C}\rightarrow \mathcal{D}$ be a $\mathbb{K}$-linear functor between two $\mathbb{K}$-linear categories. We call $\mathcal{F}$  a $\emph{covering functor}$ if the maps
    \[
        \coprod_{z/b}\mathcal{C}(x,z)\rightarrow \mathcal{D}(a,b)\;\;\;\text{and}\;\;\; \coprod_{t/a}\mathcal{C}(t,y)\rightarrow \mathcal{D}(a,b)
    \]
    which are induced by $\mathcal{F}$, are bijective for any two objects $a$ and $b$ of $\mathcal{D}$. Here $t$ and $z$ range over all objects of $\mathcal{C}$ such that $\mathcal{F}t=a$ and $\mathcal{F}z=b$, respectively; the maps are also supposed to be bijective for all $x$ and $y$ chosen such that $\mathcal{F}x=a$ and $\mathcal{F}y=b$, respectively.
\end{defn}

Given a quiver $Q$ and an admissible ideal $I$ in the $\mathbb{K}$-linear category $\mathbb{K}Q$, let $\mathcal{A}$ be the $\mathbb{K}$-linear category $\mathbb{K}Q/I$. The left $\mathcal{A}$-module $M$ (or representation of $\mathcal{A}$) is called $\emph{locally-finite dimensional}$ if $M(x)$ is finite dimensional, for all objects $x \in \mathcal{A}$. We denote by $\mathcal{A}\MOD$ the category of all left $\mathcal{A}$-modules, by $\mathcal{A}\operatorname{-\,Mod}$ (resp., $\mathcal{A}\mod$) the full subcategory formed by all locally-finite dimensional (resp., finite dimensional) left $\mathcal{A}$-modules. Denote by $\mathsf{proj}\,\mathcal{A}$ the subcategory of $\mathcal{A}\mod$ formed by projective $\mathcal{A}$-modules. Given a vertex $x\in Q_{0}$, $P_{x}$ denotes the indecomposable projective module in $\mathsf{proj}\,\mathcal{A}$ at vertex $x$.

Let $p:(\Gamma,J)\rightarrow(Q,I)$ be a Galois covering of bound quivers, with $G$ as a group of automorphisms of $(\Gamma,J)$. If $\mathcal{B} \coloneqq \mathbb{K}\Gamma/J$ and $\mathcal{A}\coloneqq\mathbb{K}Q/I$, then $p$ induces naturally a covering functor between $\mathbb{K}$-linear categories $\mathcal{F} : \mathcal{B}\rightarrow \mathcal{A}$, and following \cite{Bongartz-Gabriel}, we obtain two new functors:

\begin{enumerate}
    \item The $\emph{pull-up functor}$ $\mathcal{F}_{\bullet} : \mathcal{A}\MOD\rightarrow \mathcal{B}\MOD$, given by $\mathcal{F}_{\bullet}(M) = M \circ \mathcal{F}$, for any $M\in \mathcal{A}\MOD$, and
    \item the $\emph{push-down functor}$ $\mathcal{F}_{\lambda}: \mathcal{B}\MOD\rightarrow \mathcal{A}\MOD$, the left adjoint to $\mathcal{F}_{\bullet}$ defined by the conditions: $\mathcal{F}_{\lambda}$ is right exact and $\mathcal{F}_{\lambda}P_{\hat{x}}=P_x$, for any vertices $\hat{x} \in\Gamma_{0} ,x\in Q_0$ with $\mathcal{F}\hat{x} =x$.
\end{enumerate}

Given $M\in\mathcal{B}\MOD$, we get $\mathcal{F}_{\lambda}M\in \mathcal{A}\MOD$ in the following way: For each object $x\in \mathcal{A}$, we set

\[
    \mathcal{F}_\lambda M(x)=\bigoplus_{\hat{x}/x}M(\hat{x})
\]
where $\hat{x}$ ranges over all objects of $\mathcal{B}$ such that $\mathcal{F}\hat{x}=x$. For an arrow 
\begin{tikzcd}
    x \arrow[r,"\alpha"] & y
\end{tikzcd}
in $\mathcal{A}$, the map $\mathcal{F}_\lambda M(\alpha):\mathcal{F}_\lambda M(x)\rightarrow\mathcal{F}_\lambda M (y)$ assigns to $(v_{\hat{x}})\in\bigoplus\limits_{\hat{x}/x}M(\hat{x})$ the family $\sum\limits_{\hat{x}}M(\leftindex_{\hat{x}}{\widehat{\alpha}}_{\hat{y}})(v_{\hat{x}})\in\bigoplus\limits_{\hat{y}/y}M(\hat{y})$, where $\leftindex_{\hat{x}}{\widehat{\alpha}}_{\hat{y}}$ runs through all the preimages of $\alpha$.

\begin{rmk}
    Note that the pull-up functor $\F_{\bullet}$ restricts to locally finite-dimensional modules, and the push-down functor $\F_{\lambda}$ restricts to finite-dimensional modules.
\end{rmk}

We will denote by $\mathcal{B}\ind$ the full subcategory of $\mathcal{B}\mod$ formed by choosing a representative of each isomorphism class of finite-dimensional indecomposable $\mathcal{B}$-modules. We will choose these representatives in such a way so that if $T\in \mathcal{B}\ind$, then $T^{\,g}\in\mathcal{B}\ind$, for each $g\in G$. Similarly, we let $\mathcal{A}\ind$ denote the full subcategory of $\mathcal{A}\mod$ formed by choosing a representative of each isomorphism class of finite-dimensional indecomposable $\mathcal{A}$-modules.

\begin{enumerate}
    \item $\F_{\lambda}$ and $\F_{\bullet}$ are exact functors and send indecomposable projective modules to indecomposable projective modules
    \item If the push-down functor $\mathcal{F}_{\lambda}$ preserves indecomposable modules, then $\mathcal{F}_{\lambda}$ induces an injection $\mathcal{F}_{\lambda}:\mathcal{B}\ind/G\hookrightarrow \mathcal{A}\ind$.
    \item If $G$ is torsion-free, then $\mathcal{F}_{\lambda}$ preserves indecomposable modules. More generally, if no element of $G$ fixes $M\in\mathcal{B}\MOD$, then $F_\lambda(M)$ is indecomposable.
    \item $\F_{\bullet}\F_{\lambda}\cong \bigoplus_{g\in G} \mathrm{Id}_{\,\mathcal{B}\MOD}^{\,g}$
    \item If $N\in \mathcal{A}\ind$ and $M\in \mathcal{B}\MOD$ are such that $\F_{\lambda}M=N$, then $M\in\mathcal{B}\ind$.
    \item If $N\in \mathcal{A}\ind$ and $M_1,M_2\in \mathcal{B}\MOD$ are such that $\F_{\lambda}M_{1}=N=\F_{\lambda}M_{2}$, then $M_{1}\cong M_{2}^{\,g}$, for some $g\in G$.
\end{enumerate}

\subsection{Basic notions of \texorpdfstring{$\tau$}{tau}-tilting theory}

Let $\mathcal{A}$ be a finite dimensional $\mathbb{K}$-algebra. For $M\in \mathcal{A}\mod$, we will use $\vert M\rvert$ to denote the number of non-isomorphic indecomposable direct summands of $M$. We denote the Auslander-Reiten translate of $\mathcal{A}\mod$ by $\tau_{\mathcal{A}}$.

We recall some basic definitions and some important properties of support $\tau_{\mathcal{A}}$-tilting modules from \cite{AIR}.

\begin{defn} Let $M\in \mathcal{A}\mod$. Then $M$ is said to be
    \begin{enumerate}
        \item  $\tau_{\mathcal{A}}$-$\emph{rigid}$ if $\Hom_{\mathcal{A}}(M,\tau_{\mathcal{A}} M)=0$.
        \item  $\tau_{\mathcal{A}}$-$\emph{tilting}$ (resp., $\emph{almost complete }\tau_{\mathcal{A}}$-$\emph{tilting}$ module) if $M$ is $\tau_{\mathcal{A}}$-rigid and $\lvert M\rvert=\lvert \mathcal{A}\rvert$ (resp., $\lvert M\rvert=\lvert \mathcal{A}\rvert-1$).
        \item $\emph{support }\tau_{\mathcal{A}}$-$\emph{tilting}$ if there exists an idempotent $e$ of $\mathcal{A}$ such that $M$ is a $\tau_{\mathcal{A}}$-tilting $\mathcal{A}/\langle e\rangle$-module.
    \end{enumerate}
\end{defn}

We can view them as certain pairs of $\mathcal{A}$-modules.

\begin{defn}
\phantomsection\label{supptautiltmod}
    Let $(M,P)$ be a pair with $M\in \mathcal{A}\mod$ and $P\in \mathsf{proj}\,\mathcal{A}$. The pair $(M,P)$ is called
    \begin{enumerate}
        \item a $\emph{$\tau_{\mathcal{A}}$-rigid pair}$ if $M$ is $\tau_{\mathcal{A}}$-rigid and $\Hom_{\mathcal{A}}(P,M)=0$.
        \item a \emph{support $\tau_{\mathcal{A}}$-tilting pair} (resp., \emph{almost complete support $\tau_{\mathcal{A}}$-tilting pair}) if $(M,P)$ is $\tau_{\mathcal{A}}$-rigid and $\lvert M\rvert+\lvert P\rvert=\lvert \mathcal{A}\rvert$ (resp., $\lvert M\rvert+\lvert P\rvert=\lvert \mathcal{A}\rvert-1$).
    \end{enumerate}
\end{defn}

\begin{defn}
    Let $M,M'\in \mathcal{A}\mod$ and $P,P'\in \mathsf{proj}\,\mathcal{A}$.
    \begin{enumerate}
        \item A module $X\in \mathcal{A}\mod$ is called $\emph{basic}$ if every indecomposable direct summand of $X$ occurs only once in the direct sum decomposition of $X$ into indecomposables.
        \item We say that $(M, P )$ is $\emph{basic}$ if $M$ and $P$ are basic.
        \item We say that $(M, P )$ is a direct summand of $(M',P')$ if $M$ is a direct summand of $M'$ and $P$ is a direct summand of $P'$.
        \item We say that $(M, P )$ is isomorphic to $(M',P')$ if $M$ is isomorphic to $M'$ and $P$ is isomorphic to $P'$.
    \end{enumerate}
\end{defn}

\begin{rmk}
    Given a support $\tau_{\mathcal{A}}$-tilting pair $(M,P)$ with $P$ basic, $M$ determines $P$ uniquely up to isomorphism, see \cite[Proposition 2.3]{AIR}.
\end{rmk}

We will use $s\tau_{\mathcal{A}}\tilt$ to denote the set of all basic support $\tau_{\mathcal{A}}$-tilting pairs (up to isomorphism). We have the following analog of Bongartz completion for tilting modules.

\begin{thm}[see Theorem 2.10 \cite{AIR}]
\label{tau-rigid}
    Any $\tau_{\mathcal{A}}$-rigid $\mathcal{A}$-module is a direct summand of some $\tau_{\mathcal{A}}$-tilting $\mathcal{A}$-module.
\end{thm}

It turns out that support $\tau_{\mathcal{A}}$-tilting modules (pairs) have exactly two complements. This leads to the notion of mutations.

\begin{thm}[see Theorem 2.18 \cite{AIR}]
    Let $\mathcal{A}$ be a finite dimensional $\mathbb{K}$-algebra. Then any basic almost complete support $\tau_{\mathcal{A}}$-tilting pair for $\mathcal{A}$, up to isomorphism, is a direct summand of exactly two basic support $\tau_{\mathcal{A}}$-tilting pairs.
\end{thm}

These two support $\tau_{\mathcal{A}}$-tilting pairs are said to be mutations of each other.

\begin{defn}
    Let $(M,P)$ and $(M',P')$ be two non-isomorphic basic support $\tau_{\mathcal{A}}$-tilting pairs for $\mathcal{A}$. They are said to be $\emph{mutations}$ of each other if there exists a basic almost complete support $\tau_{\mathcal{A}}$-tilting pair $(T,Q)$ which is a direct summand of $(M,P)$ and $(M',P')$. 
    
    If $X$ is an indecomposable $\mathcal{A}$-module satisfying either $M = T\oplus X$ and $P = Q$; or $P = Q \oplus X$ and $M = T$, then we write $(M',P') = \mu_{X}(M,P)$ or simply $M' = \mu_{X}(M)$.
\end{defn}

For $M\in \mathcal{A}\mod$, we denote by $\mathsf{Fac}\,M$, the subcategory of $\mathcal{A}\mod$ of all factor modules of finite direct sums of copies of $M$. Let $\operatorname{f-tors}\mathcal{A}$ denote the set of functorially finite torsion classes in $\mathcal{A}\mod$.

\begin{thm}[see Theorem 2.7 \cite{AIR}]
    There is a bijection
    \[
    \begin{array}{cccc}
         s\tau_{\mathcal{A}}\tilt & \leftrightarrow & \operatorname{f-tors}\mathcal{A} & \text{given by}\\
         M & \mapsto & \mathsf{Fac}\,M & \;\\
         \mathcal{P}(\mathcal{T}) & \mapsfrom & \mathcal{T} & \;
    \end{array}
    \]
    where $\mathcal{P}(\mathcal{T})$ is the direct sum of one copy of each of the indecomposable $\Ext$-projective objects in $\mathcal{T}$ (up to isomorphism).
\end{thm}

Using the above bijection, we get that the inclusion in $\operatorname{f-tors}\mathcal{A}$ gives a partial order on $s\tau_{\mathcal{A}}\tilt$.

\begin{defn}
    Let $M = T \oplus X$ and $M'$ be support $\tau_{\mathcal{A}}$-tilting $\mathcal{A}$-modules such that $M' = \mu_{X}(M)$ for some indecomposable $\mathcal{A}$-module $X$.
    \begin{enumerate}
        \item If $\mathsf{Fac}\,M\supseteq \mathsf{Fac}\,M'$, then we call $M'$ the $\emph{left mutation}$ of $M$, and we denote it by $M' = \mu_{X}^{-}(M)$
        \item If $\mathsf{Fac}\,M\subseteq \mathsf{Fac}\,M'$, then we call $M'$ the $\emph{right mutation}$ of $M$, and we denote it by $M' = \mu_{X}^{+}(M)$
    \end{enumerate} 
\end{defn}

We can define a quiver that records the information about left and right mutations.

\begin{defn}
\phantomsection\label{MutQui}
    The \emph{mutation quiver} $Q(s\tau_{\mathcal{A}}\tilt)$ of $\mathcal{A}$ is defined as follows:
    \begin{enumerate}
        \item The set of vertices are the elements of $s\tau_{\mathcal{A}}\tilt$
        \item We draw an arrow from $M$ to $M'$ if $M'$ is a left mutation of $M$.
    \end{enumerate}
\end{defn}

\begin{defn}
    If $\mathcal{A}$ is a finite dimensional $\mathbb{K}$-algebra, then $\mathcal{A}$ is called \emph{$\tau_{\mathcal{A}}$-tilting finite} if there are only finitely many isomorphism classes of indecomposable $\tau_{\mathcal{A}}$-rigid $\mathcal{A}$-modules. Otherwise, $\mathcal{A}$ is said to be \emph{$\tau_{\mathcal{A}}$-tilting infinite}.
\end{defn}

\subsection{Conventions}

We will use $\mathcal{A}$ to denote a finite-dimensional, connected, and basic $\mathbb{K}$-algebra given by a finite and connected bound quiver $(Q,I)$ with $I$ an admissible ideal in the path algebra $\mathbb{K}Q$. We let $p : (\Gamma,J)\rightarrow (Q,I)$ denote a Galois covering with $G$ as a group of automorphism of $(\Gamma,J)$. We also assume $G$ to be torsion-free.

We will use $\mathcal{F} : \mathcal{B} \coloneqq \mathbb{K}\Gamma/J\rightarrow \mathcal{A}\coloneqq\mathbb{K}Q/I$ to denote the covering functor between the two $\mathbb{K}$-linear categories, and $\F_{\bullet}\,, \F_{\lambda}$ to denote the pull-up and the push-down functor associated to the above Galois cover, respectively. We will denote by $\mathcal{A}\mod_1$ the full subcategory of $\mathcal{A}\mod$ of modules of first kind with respect to $\F_{\lambda}$. Since $G$ is torsion-free, we know that $\F_{\lambda}$ will preserve indecomposable modules.

Note that $\Gamma$ will be locally finite and hence, $\mathcal{B}$ is a locally bounded $\mathbb{K}$-linear category. By \cite[Theorem, Section 3.6, page 89]{GabUnivCov}, we know that the Auslander-Reiten translate exists for $\mathcal{B}\mod$. We will denote the Auslander-Reiten translate for $\mathcal{B}\mod$ and $\mathcal{A}\mod$ by $\tau_{\mathcal{B}}$ and $\tau_{\mathcal{A}}$, respectively.

As mentioned in Section \ref{intro}, by a subcategory $\mathcal{C}\subseteq \mathcal{B}\mod$ (resp. $\mathcal{D}\subseteq \mathcal{A}\mod$) of a $\mathbb{K}$-linear category we mean a subcategory of $\mathcal{B}\mod$ (resp. $\mathcal{A}\mod$) that is full, skeletal, and closed under finite direct sums and direct summands.

\section{Mutations of \texorpdfstring{$G$}{G}-orbits}

Let $\mathcal{C}\subseteq\mathcal{B}\mod$ be a subcategory. We will say that $\mathcal{C}$ is $\emph{$G$-equivariant}$ if, for every $\mathcal{B}$-module $M\in\mathcal{C}$, the module $M^g\in\mathcal{C}$ for every $g \in G$. For a $G$-equivariant subcategory $\mathcal{C}$ we write $\lvert\mathcal{C}\rvert_{G}$ to denote the number of $G$-orbits of distinct indecomposable objects in $\mathcal{C}$. For subcategories $\mathcal{C}_1,\mathcal{C}_2\subseteq\mathcal{B}\mod$, we denote by $\mathsf{add}(\mathcal{C}_1,\mathcal{C}_2)$ the smallest (additive) subcategory of $\mathcal{B}\mod$ containing both $\mathcal{C}_1$ and $\mathcal{C}_2$, and by $\mathsf{add}_{G}(\mathcal{C}_1,\mathcal{C}_2)$ the smallest $G$-equivariant (additive) subcategory of $\mathcal{B}\mod$ containing both $\mathcal{C}_1$ and $\mathcal{C}_2$. We define the notion of $(G,\tau_{\mathcal{B}})$-rigidity and $(G,\tau_{\mathcal{B}})$-tilting over a Galois covering in the following manner:

\begin{defn}
\phantomsection\label{(G,tau)-tilting}

    Let $M\in\mathcal{B}\mod$ and $\mathcal{T}$ be a subcategory in $\mathcal{B}\mod$.
    \begin{enumerate}
        \item We say $M$ is $\emph{$(G,\tau_{\mathcal{B}})$-rigid}$ if $\Hom_{\mathcal{B}}(M,\tau_{\mathcal{B}}M ^{g})=0$, for every $g\in G$.
        \item We say $\mathcal{T}$ is $\emph{$\tau_{\mathcal{B}}\,$-rigid}$ if $\Hom_{\mathcal{B}}(X,\tau_{\mathcal{B}}Y)=0$, for every $X,Y\in\mathcal{T}$.
        \item We say $\mathcal{T}$ is $\emph{$(G,\tau_{\mathcal{B}})$-tilting}$ if the following three conditions are satisfied:
        \begin{enumerate}
            \item $\mathcal{T}$ is $\tau_{\mathcal{B}}$-rigid
            \item $\mathcal{T}$ is $G$-equivariant.
            \item $\lvert\mathcal{T}\rvert_{G}=\lvert \mathcal{A}\rvert$
        \end{enumerate}
    \end{enumerate}
\end{defn}

\begin{rmk}
    The above definition of $(G,\tau_{\mathcal{B}})$-rigid and $(G,\tau_{\mathcal{B}})$-tilting is mainly motivated by Proposition \ref{NoHom} in the next section. Moreover, notice that given a $\tau_{\mathcal{A}}$-tilting $\mathcal{A}$-module, say $T$, the subcategory $\mathsf{add}(T)\subseteq \mathcal{A}\mod$ has exactly $\lvert\mathcal{A}\rvert$ many distinct indecomposable objects. Hence, we want the subcategory $\F_{\lambda}\mathcal{T}\subseteq \mathcal{A}\mod$ to also have exactly $\lvert \mathcal{A}\rvert$ many distinct indecomposable objects.
\end{rmk}

Similar to Definition \ref{supptautiltmod} we can define support $(G,\tau_{\mathcal{B}})$-tilting subcategories over a Galois covering.

\begin{defn}
\phantomsection\label{supp(G,tau)-tilting}
    Let $(\mathcal{T},\mathcal{P})$ be a pair where $\mathcal{T}\subseteq\mathcal{B}\mod$ and $\mathcal{P}\subseteq\mathsf{proj}\,\mathcal{B}$ are  subcategories
    \begin{enumerate}
        \item The pair $(\mathcal{T},\mathcal{P})$ is called \emph{$(G,\tau_{\mathcal{B}})$-rigid pair} if the following conditions are satisfied:
        \begin{enumerate}
            \item $\mathcal{T}$ is $\tau_{\mathcal{B}}\,$-rigid
            \item Both $\mathcal{T}$ and $\mathcal{P}$ are $G$-equivariant
            \item $\Hom_{\mathcal{B}}(\mathcal{P},\mathcal{T})=0$
        \end{enumerate}
        \item The pair $(\mathcal{T},\mathcal{P})$ is called a \emph{support $(G,\tau_{\mathcal{B}})$-tilting pair} (resp., \emph{almost complete support $(G,\tau_{\mathcal{B}})$-tilting pair}) if $(\mathcal{T},\mathcal{P})$ is $(G,\tau_{\mathcal{B}})$-rigid pair and $\lvert \mathcal{T}\rvert_{G}+\lvert \mathcal{P}\rvert_{G}=\lvert \mathcal{A}\rvert$ (resp., $\lvert \mathcal{T}\rvert_{G}+\lvert \mathcal{P}\rvert_{G}=\lvert \mathcal{A}\rvert-1$).
    \end{enumerate}
\end{defn}

\begin{defn}
\phantomsection\label{equivsubcatpair}
    Let $\mathcal{T},\mathcal{T}'\subseteq\mathcal{B}\mod$ and $\mathcal{P},\mathcal{P}'\subseteq\mathsf{proj}\,\mathcal{B}$ be subcategories.
    \begin{enumerate}
        \item We say two pairs $(\mathcal{T},\mathcal{P})$ and $(\mathcal{T}',\mathcal{P}')$ are $\emph{equivalent}$, denoted by $(\mathcal{T},\mathcal{P})\simeq(\mathcal{T}',\mathcal{P}')$, if $\mathcal{T}$ and $\mathcal{T}'$ are equivalent as categories, and $\mathcal{P}$ and $\mathcal{P}'$ are equivalent as categories. Otherwise, they are $\emph{non-equivalent}$, denoted by $(\mathcal{T},\mathcal{P})\nsimeq(\mathcal{T}',\mathcal{P}')$ .
        \item We write $(\mathcal{T}',\mathcal{P}')\preceq(\mathcal{T},\mathcal{P})$ if $\mathcal{T}'\subseteq\mathcal{T}$ and $\mathcal{P}'\subseteq\mathcal{P}$.
    \end{enumerate}
\end{defn}

We will use $s(G,\tau_{\mathcal{B}})\,\tilt$ to denote the set of all support $(G,\tau_{\mathcal{B}})$-tilting pair (up to equivalence).

\begin{rmk}
\label{pair of subcategories instead of pair of mod}
    Note that we can also interpret the elements of $s\tau_{\mathcal{A}}\tilt$ as a pair of subcategories of $\mathcal{A}\mod$, instead of just a pair of $\mathcal{A}$-modules. In other words, given a basic support $\tau_{\mathcal{A}}$-tilting pair $(T,P)$, we consider the pair $(\mathsf{add}(T),\mathsf{add}(P))$.
\end{rmk}

Using above notions we can mutate at the level of $G$-orbits.

\begin{defn}
\phantomsection\label{(G,tau)mut}
    Let $(\mathcal{T}_1,\mathcal{P}_1),(\mathcal{T}_2,\mathcal{P}_2)$ be two non-equivalent support $(G,\tau_{\mathcal{B}})$-tilting pairs. They are said to be $\emph{$(G,\tau_{\mathcal{B}})$-mutations}$ of each other if there exists an almost complete support $(G,\tau_{\mathcal{B}})$-tilting pair $(\mathcal{T},\mathcal{Q})$ such that $(\mathcal{T},\mathcal{Q})\preceq(\mathcal{T}_1,\mathcal{P}_1)$ and $(\mathcal{T},\mathcal{Q})\preceq(\mathcal{T}_2,\mathcal{P}_2)$.

    If $X\in\mathcal{B}\ind$ satisfies either $\mathcal{T}_1=\mathsf{add}_{G}(\mathcal{T},X)$ and $\mathcal{P}_1 = \mathcal{Q}$; or $\mathcal{P}_1=\mathsf{add}_{G}(\mathcal{Q},X)$ and $\mathcal{T}_1 = \mathcal{T}$, then we write $\mu_{\mathcal{O}(X)}(\mathcal{T}_1)=\mathcal{T}_2$ to denote the mutation of $(\mathcal{T}_1,\mathcal{P}_1)$ at the $G$-orbit $\mathcal{O}(X)$.
\end{defn}

\section{Mutations via the push-down functor}

The proposition below shows that the above notion of $(G,\tau_{\mathcal{B}})$-rigid in $\mathcal{B}\mod$ agrees with the notion of $\tau_{\mathcal{A}}$-rigid in $\mathcal{A}\mod$, via the push-down functor.

\begin{prop}
\phantomsection\label{NoHom}
    Let $M,N\in\mathcal{B}\mod$. Then
    \[
        \Hom_{\mathcal{A}}(\F_{\lambda}M\,,\, \tau_{\mathcal{A}}\F_{\lambda}N)=0\;\text{if and only if}\; \Hom_{\mathcal{B}}(M\,,\, \tau_{\mathcal{B}}N^{g})=0,\;\text{for every}\;g\in G.
    \]
\end{prop}

\begin{proof}
    By \cite[Theorem, Section 3.6, page 89]{GabUnivCov}, we know that the Auslander-Reiten translate commutes with push-down functor. Hence, we get
    \[
        \Hom_{\mathcal{A}}(\F_{\lambda}M\,,\, \tau_{\mathcal{A}}\F_{\lambda}N)\cong \Hom_{\mathcal{A}}(\F_{\lambda}M\,,\, \F_{\lambda}\tau_{\mathcal{B}}N)
    \]
    Since $(\F_{\lambda},\F_{\bullet})$ is an adjoint pair, we get,
    \[
        \Hom_{\mathcal{A}}(\F_{\lambda}M\,,\, \F_{\lambda}\tau_{\mathcal{B}}N) \cong \Hom_{\mathcal{B}}(M\,,\, \F_{\bullet}\F_{\lambda}\tau_{\mathcal{B}}N)
         \cong \Hom_{\mathcal{B}}(M\,,\, \oplus_{g\in G}\,(\tau_{\mathcal{B}}N)^{g})
    \]
    Since $M$ is a finite dimensional $\mathcal{B}$-module and hence, a finitely generated $\mathcal{B}$-module, we have
    \[
        \Hom_{\mathcal{B}}(M\,,\, \oplus_{g\in G}(\tau_{\mathcal{B}}N)^{g}) \cong \bigoplus_{g\in G}\Hom_{\mathcal{B}}(M\,,\, (\tau_{\mathcal{B}}N)^{g})
    \]
    which is, indeed, isomorphic to $\bigoplus_{g\in G}\Hom_{\mathcal{B}}(M\,,\, \tau_{\mathcal{B}}N^{g})$. 
    
    Hence, we get
    \[
        \Hom_{\mathcal{A}}(\F_{\lambda}M\,,\, \tau_{\mathcal{A}}\F_{\lambda}N)\cong \bigoplus_{g\in G}\Hom_{\mathcal{B}}(M\,,\, \tau_{\mathcal{B}}N^{g})
    \]
    which gives the desired result.
\end{proof}

The previous proposition yields the following.

\begin{cor}
    A module $M\in\mathcal{B}\mod$ is $(G,\tau_{\mathcal{B}})$-rigid if and only if $\mathcal{F}_{\lambda}M$ is $\tau_{\mathcal{A}}$-rigid.
\end{cor}

\begin{rmk} 
A $\tau_{\mathcal{B}}$-rigid module may fail to be sent to a $\tau_\mathcal{A}$-rigid module by the push-down functor. An example of this is when taking $\mathcal{A} = k[x]/\langle x^2\rangle$ the algebra of dual numbers, and take $\mathcal{B}$ given by the $A_\infty^\infty$ quiver which is a double-infinite path with all paths of length two in the ideal. The simple $\mathcal{B}$-modules are all $\tau_{\mathcal{B}}$-rigid while their image via the push down functor is the simple $\mathcal{A}$-module, which is not $\tau_\mathcal{A}$-rigid. We note that a simple $\mathcal{B}$-module, while being $\tau_{\mathcal{B}}$-rigid, is not $(G, \tau_{\mathcal{B}})$-rigid.
\end{rmk} 

We would like to understand how the push-down functor behaves with the notion of $(G,\tau_{\mathcal{B}})$-mutation of $G$-orbits and mutation of $\tau_{\mathcal{A}}$-tilting modules. Given a $\tau_{\mathcal{A}}$-tilting $\mathcal{A}$-module, we first consider a left mutation of it. We first recall some definitions and prove some technical results.

\begin{defn}
    Let $\mathcal{C}\subseteq \mathcal{A}\mod$ be a subcategory and $V$ an $\mathcal{A}$-module.
    \begin{enumerate}
        \item A morphism $f: V\rightarrow X$, with $X\in\mathcal{C}$, is called a $\emph{left $\mathcal{C}$ approximation}$ of $V$ if the map 
        \[
            \begin{array}{cccccc}
                 f_{*} & : & \Hom_{\mathcal{A}}(X,X') & \rightarrow & \Hom_{\mathcal{A}}(V,X') &  \\
                 & & h & \mapsto & h\circ f
            \end{array}
        \]
        is surjective.
        \item A morphism $f:V\rightarrow W$ in $\mathcal{A}\mod$ is called $\emph{left minimal}$ if, whenever there is a morphism $\phi:W\rightarrow W$ such that $f=\phi f$, then $\phi$ must be an isomorphism.
        \[
            \begin{tikzcd}
	           {V} && {W} \\
	           \\
	           && {W}
	           \arrow["{\phi}", from=1-3, to=3-3]
	           \arrow["f"', from=1-1, to=3-3]
	           \arrow["{f}", from=1-1, to=1-3]
            \end{tikzcd}
        \]
        Moreover, if $f$ is also a left $\mathcal{C}$ approximation of $V$, then we call $f$ a $\emph{minimal left $\mathcal{C}$}$ $\emph{approximation}$ of $V$.
    \end{enumerate}
\end{defn}

Below we record some well known facts about left minimal morphisms. Let $\mathcal{R}$ be any locally bounded $\mathbb{K}$-linear category.

\begin{lem}
    Let $f:V\rightarrow W$ be a morphism in $\mathcal{R}\mod$. The following statements are equivalent:
    \begin{enumerate}
        \item f is left minimal.
        \item If $f$ is isomorphic to
         \[
             \begin{tikzcd}
                V \arrow[r,"\spmat{f_1\\0}"] & W_1\oplus W_2
             \end{tikzcd}
         \]
        then $W_2=0$.
        \item $\mathrm{Coker}\,f:W\rightarrow \mathrm{coker}\,f$ is a radical map.
    \end{enumerate}
\end{lem}
 
\begin{rmk}
\phantomsection\label{radicalnon-iso}
    If $\varphi:M\rightarrow N$ is a morphism in $\mathcal{R}\mod$, with $M=M_1\oplus M_2\oplus \dotsb\oplus M_k$ and $N=N_1\oplus N_2\oplus \dotsb\oplus N_l$ the decompositions of $M$ and $N$ into indecomposables, and if $\varphi$ is represented by the matrix $(\varphi_{ij})_{l\times k}$, then $\varphi$ is a radical map if and only if each $\varphi_{ij}$ is a non-isomorphism.
\end{rmk}

\begin{prop}
\phantomsection\label{leftapprx}
    Let $M,U, U'\in \mathcal{B}\mod$. If 
    \begin{tikzcd}
        M \arrow[r,"f"] & U'
    \end{tikzcd}
    is a left $\add_{G}(U)$ approximation of $M$, then
    \begin{tikzcd}
        \F_{\lambda} M \arrow[r,"\F_{\lambda}f"] & \F_{\lambda}U'
    \end{tikzcd}
    is a left $\add(\F_{\lambda}U)$ approximation of $\F_{\lambda} M$.
\end{prop}
\begin{proof}
    We want to show that for any 
    \begin{tikzcd}
        \F_{\lambda}M \arrow[r,"p"] & L,
    \end{tikzcd}
    where $L\in\add(\F_{\lambda}U)$, there exists 
    \begin{tikzcd}
        \F_{\lambda}U' \arrow[r,"h"] & L
    \end{tikzcd}
    such that the following diagram commutes
        \begin{equation}
        \label{eqn1}
            \begin{tikzcd}
	           {\F_{\lambda}M} && {\F_{\lambda}U'} \\
	           \\
	           && {L}
	           \arrow["h", dotted, from=1-3, to=3-3]
	           \arrow["p"', from=1-1, to=3-3]
	           \arrow["{\F_{\lambda}f}", from=1-1, to=1-3]
            \end{tikzcd}
        \end{equation}
    We can assume that $L=\F_{\lambda}W'$ for some $W'\in\add_{G}(U)$. Recall that
    
        \begin{align*}
            \bigoplus_{g\in G}\mathrm{Hom}_{\mathcal{B}}\left(M,(W')^{g}\right) & \cong  \mathrm{Hom}_{\mathcal{A}}\left(\F_{\lambda}M,\F_{\lambda}W'\right)\\
            \left(p_{g}\right)_{g\in G} & \mapsto \sum\F_{\lambda}p_{g}.
        \end{align*}
    
    Suppose $$p=\F_{\lambda}p_{g_1}+\F_{\lambda}p_{g_2}+\dotsc+\F_{\lambda}p_{g_\ell}.$$   Since $W'\in\add_{G}(U)$, we have $(W')^{g_1},(W')^{g_2},\dotsc,(W')^{g_\ell}\in\add_{G}(U)$. Using the fact that
    \begin{tikzcd}
        M \arrow[r,"f"] & U'
    \end{tikzcd}
    is a left $\add_{G}(U)$ approximation, there exists
    \begin{tikzcd}
        U' \arrow[r,"H"] & \bigoplus_{i=1}^{\ell}(W')^{g_i}
    \end{tikzcd}
    such that we get the following commutative diagram
    \[
        \begin{tikzcd}
	       M && U' \\
	       \\
	       \\
	       && {\bigoplus_{i=1}^{\ell}(W')^{g_i}}
	       \arrow["f", from=1-1, to=1-3]
	       \arrow["{\spmat{p_{g_1}\\p_{g_2}\\\vdots\\p_{g_{\ell}}}}"'{pos=0.4}, from=1-1, to=4-3]
	       \arrow["{H=\spmat{H_1\\H_2\\\vdots\\H_{\ell}}}", dotted, from=1-3, to=4-3]
        \end{tikzcd}
    \]
    Now we apply the push-down functor to the above commutative diagram to get the following commutative diagram
    \[
        \begin{tikzcd}
	       {\F_{\lambda}M} && {\F_{\lambda}U'} \\
	       \\
	       \\
	       && {L^{\oplus\ell}}
	       \arrow["{\F_{\lambda}f}", from=1-1, to=1-3]
	       \arrow["{\spmat{\F_{\lambda}p_{g_1}\\\F_{\lambda}p_{g_2}\\\vdots\\\F_{\lambda}p_{g_{\ell}}}}"'{pos=0.4}, from=1-1, to=4-3]
	       \arrow["{H=\spmat{\F_{\lambda}H_1\\\F_{\lambda}H_2\\\vdots\\\F_{\lambda}H_{\ell}}}", dotted, from=1-3, to=4-3]
        \end{tikzcd}
    \]
    By the commutativity of the above diagram we get that
    \[
        p=\F_{\lambda}p_{g_1}+\F_{\lambda}p_{g_2}+\dotsc+\F_{\lambda}p_{g_\ell}=\left(\F_{\lambda}H_1+\F_{\lambda}H_2+\dotsc+\F_{\lambda}H_\ell\right)\F_{\lambda}f
    \]
    Hence, $h\coloneqq\F_{\lambda}H_1+\F_{\lambda}H_2+\dotsc+\F_{\lambda}H_\ell$ is the desired $h$ in (\ref{eqn1}).
\end{proof}

\begin{lem}
\phantomsection\label{gRKS}
    Let $M_1,M_2,\dotsc,M_{l}\in\mathcal{B}\ind$ and $L\in\mathcal{B}\mod$. If $\bigoplus_{k=1}^{l}\F_{\lambda}M_k\cong\F_{\lambda} L$, then $L\cong \bigoplus_{k=1}^{l}M_{k}^{\,g_{k}}$, where $g_{k}\in G,1\leq k\leq l$.
\end{lem}

\begin{proof}
    Let $L\cong L_{1}\oplus L_{2}\oplus\dotsb\oplus L_{r}$ with each $L_{t}\in \mathcal{B}\ind$ be a decomposition of $L$ into indecomposable $\mathcal{B}$-modules. Therefore, we get
    \[
        \bigoplus_{t=1}^{r}\F_{\lambda}L_{t}\cong\bigoplus_{k=1}^{l}\F_{\lambda}M_k
    \]
    Since $G$ is torsion-free, $\F_{\lambda}$ sends a finite-dimensional indecomposable $\mathcal{B}$-module to a finite-dimensional indecomposable $\mathcal{A}$-module. By the Krull-Remak-Schmidt theorem, $l=r$ and, without loss of generality, we have $\F_{\lambda}L_{k}\cong\F_{\lambda}M_{k}$, for all $1\leq k\leq l$. Hence, we conclude each $L_k\cong M_{k}^{\,g_{k}}$, for some $g_{k}\in G$ and for all $1\leq k\leq l$. Therefore, $L\cong \bigoplus_{k=1}^{l}M_{k}^{\,g_{k}}$.
\end{proof}

\begin{lem}
\phantomsection\label{non-isotonon-iso}
    Let $\psi:U\rightarrow V$ be a morphism in $\mathcal{B}\mod$. If $\psi$ is a non-isomorphism in $\mathcal{B}\mod$, then $\F_{\lambda}\psi:\F_{\lambda}U\rightarrow \F_{\lambda}V$ is also a non-isomorphism in $\mathcal{A}\mod$.
\end{lem}

\begin{proof}
    Assume that $\psi$ is a non-isomorphism in $\mathcal{B}\mod$. Then either $\mathrm{ker}\,\psi\neq0$ or $\mathrm{coker}\,\psi\neq0$. Therefore, we have one the following two exact sequences
    \[
        \begin{tikzcd}
            0\arrow[r] & \mathrm{ker}\,\psi \arrow[r] & U \arrow[r,"{\psi}"] & V
        \end{tikzcd}
        \;\;\;\;\;\text{and}\;\;\;\;\;
        \begin{tikzcd}
            U \arrow[r,"{\psi}"] & V \arrow[r] & \mathrm{coker}\,\psi \arrow[r] & 0
        \end{tikzcd}
    \]
    Now we apply the push-down functor $\F_{\lambda}$ to the above exact sequences to get that
    \[
        \text{either}\;\;0\neq\mathrm{ker}\,F_{\lambda}\psi=\F_{\lambda}\mathrm{ker}\,\psi\;\;\;\text{or}\;\;\;0\neq\mathrm{coker}\,\F_{\lambda}\psi=\F_{\lambda}\mathrm{coker}\,\psi
    \]
    which yields that $\F_{\lambda}\psi$ is an isomorphism in $\mathcal{A}\mod$.
\end{proof}

\begin{cor}
\phantomsection\label{non-radtonon-rad}
    Let $\phi:U\rightarrow V$ be a morphism in $\mathcal{B}\mod$. If $\phi$ is a radical morphism in $\mathcal{B}\mod$, then $\F_{\lambda}\phi:\F_{\lambda}U\rightarrow \F_{\lambda}V$ is also a radical morphism in $\mathcal{A}\mod$.
\end{cor}

\begin{proof}
    This follows from Remark \ref{radicalnon-iso} and the above Lemma \ref{non-isotonon-iso}.
\end{proof}

Let $T\coloneqq T_1\oplus T_2\oplus\dotsb\oplus T_i\oplus\dotsb\oplus T_n\in \mathcal{A}\mod$ be a basic $\tau_{\mathcal{A}}$-tilting module, and suppose $\mu_{T_i}^{-}(T)=T'\coloneqq T_1\oplus T_2\oplus\dotsb\oplus T_i'\oplus\dotsb\oplus T_n$, where $T'_{i}$ could be zero. Further assume that $T\in\mathcal{A}\mod_1$. Hence, there is a $\widehat{T}_{j}\in\mathcal{B}\mod$ such that $\F_{\lambda}\widehat{T}_{j}\cong T_{j}$ for every $j\in\lbrace1,2,\dotsc,n\rbrace$. The next few results will be using these notations.

\begin{lem}
\phantomsection\label{leftmin}
    If
    \begin{tikzcd}
        M \arrow[r,"f"] & N
    \end{tikzcd}
    is a left minimal morphism in $\mathcal{B}\mod$, then
    \begin{tikzcd}
        \F_{\lambda}M \arrow[r,"\F_{\lambda}f"] & \F_{\lambda}N
    \end{tikzcd}
    is also a left minimal morphism in $\mathcal{A}\mod$.
\end{lem}

\begin{proof}
    We are given that
    \begin{equation*}
    \label{eqn2}
        \begin{tikzcd}
            M \arrow[r,"f"] & N
        \end{tikzcd}
    \end{equation*}
    is a left minimal morphism in $\mathcal{B}\mod$. Therefore, the cokernel map $\mathrm{Coker}f:N\rightarrow \mathrm{coker}\,f$ is a radical map. Since $\F_{\lambda}$ is an exact functor, we have $\mathrm{coker}\,\F_{\lambda}f=\F_{\lambda}\mathrm{coker}\,f$. Therefore, by using Remark \ref{radicalnon-iso} and Lemma \ref{non-isotonon-iso} we conclude that the map $\mathrm{Coker}\,\F_{\lambda}f:\F_{\lambda}N\rightarrow \mathrm{coker}\,\F_{\lambda}f$ is a radical map. Hence,
    \begin{tikzcd}
        \F_{\lambda}M \arrow[r,"\F_{\lambda}f"] & \F_{\lambda}N
    \end{tikzcd}
    is also a left minimal morphism in $\mathcal{A}\mod$. 
\end{proof}

Using the above notations, we define 
\[
    U\coloneqq \bigoplus_{\substack{j=1\\j\neq i}}^{n}T_j\;\;\text{and}\;\; \widehat{U}\coloneqq \bigoplus_{\substack{j=1\\j\neq i}}^{n}\,\widehat{T}_{j}
\]

\begin{cor}
\phantomsection\label{minleftapprox}
    For $g\in G$, if 
    \begin{tikzcd}
        \widehat{T}_{i}^{\;g} \arrow[r,"f"] & N
    \end{tikzcd}
    is a minimal left $\add_{G}(\widehat{U})$ approximation  of $ \widehat{T}_{i}^{\;g}$ in $\mathcal{B}\mod$, then
    \begin{tikzcd}
        T_i \arrow[r,"\F_{\lambda}f"] & \F_{\lambda}N
    \end{tikzcd}
    is a minimal left $\add(U)$ approximation of $T_i$ in $\mathcal{A}\mod$.
\end{cor}

\begin{proof}
    This follows from Proposition \ref{leftapprx} and above Lemma \ref{leftmin}.
\end{proof}

The following proposition tells us that if $T\in \mathcal{A}\mod_{1}$, then $\mu_{T_i}^{-}(T)\in \mathcal{A}\mod_{1}$.

\begin{prop}
\label{LeftMutInImg}
    The indecomposable direct summand $T'_{i}$ lies in $\mathcal{A}\mod_{1}$.
\end{prop}

\begin{proof}
    If $T_{i}'=0$, then the statement is trivially true. Assume $T'_{i}\neq 0$. Observe that for every $x$ in $\Gamma$, there are only finitely many indecomposable modules in $\add_{G}(\widehat{U})$ that are supported at $x$. Hence, it is clear from this that minimal left (or right) $\add_{G}(\widehat{U})$ approximations of finite dimensional modules exist. Consider the minimal left $\add_{G}(\widehat{U})$ approximation 
    \begin{tikzcd}
        \widehat{T}_{i}^{\;g} \arrow[r,"f"] & N
    \end{tikzcd}
    as in the above Corollary \ref{minleftapprox}. Let $Y\coloneqq\mathrm{coker}(f)$. We get an exact sequence
    \begin{equation}
        \begin{tikzcd}
            \widehat{T}_{i}^{\;g} \arrow[r,"f"] & N\arrow[r] & Y\arrow[r] & 0.
        \end{tikzcd}
    \end{equation}
    Now we apply the push-down functor to the above exact sequence to get
    \begin{equation*}
        \begin{tikzcd}
            T_i \arrow[r,"\F_{\lambda}f"] & \F_{\lambda}N\arrow[r] & \F_{\lambda}Y\arrow[r] & 0.
        \end{tikzcd}
    \end{equation*}
    By Corollary \ref{minleftapprox}, we know that
    \begin{tikzcd}
        T_i \arrow[r,"\F_{\lambda}f"] & \F_{\lambda}N
    \end{tikzcd}
    is a minimal left $\add(U)$ approximation of $T_i$ in $\mathcal{A}\mod$. Hence, by using \cite[Theorem 2.30]{AIR}, we get that
    \begin{equation}
    \label{FYcongT'}
        \F_{\lambda} Y\cong (T'_{i})^{\oplus m},\;\;\text{for some}\;m\geq 1.
    \end{equation}
    Notice that since $T'_{i}\neq 0$, we have $Y\neq 0$. Now we decompose $Y$ into finite dimensional indecomposable $\mathcal{B}$-modules and then apply Krull-Remak-Schimdt theorem to the isomorphism in (\ref{FYcongT'}) to get the desired result.
\end{proof}

Notice that $\mathcal{T}_1\coloneqq\mathsf{add}_{G}(\widehat{U},\widehat{T}_i)$ is a $(G,\tau_{\mathcal{B}})$-tilting subcategory in $\mathcal{B}\mod$. Hence, we get a similar result as \cite[Theorem 2.30]{AIR}.

\begin{cor}
\phantomsection\label{LeftG-Mut}
    Consider the exact sequence
    \begin{tikzcd}
        \widehat{T}_{i}^{\;g} \arrow[r,"f"] & N\arrow[r] & Y\arrow[r] & 0
    \end{tikzcd}
    as in the previous Proposition \ref{LeftMutInImg}, where $f$ is a minimal left $\add_{G}(\widehat{U})$ approximation of $\widehat{T}_{i}^{\;g}$. Then, exactly one of the following is true:
    \begin{enumerate}
        \item $Y=0$, or
        
        \item There are $g_1, \ldots, g_k \in G$ with
        $Y\cong\bigoplus_{k=1}^{m}Z^{g_k}$ where $\mathcal{F}_\lambda Z \cong T_i'$. In this case, $\mu_{\mathcal{O}(\widehat{T}_{i})}(\mathcal{T}_{1})=\mathcal{T}_{2}\coloneqq\mathsf{add}_{G}(\widehat{U},Z)$ which is a $(G,\tau_{\mathcal{B}})$-tilting subcategory in $\mathcal{B}\mod$.
    \end{enumerate} 
\end{cor}

Now we consider right mutations of $\tau_{\mathcal{A}}$-tilting $\mathcal{A}$-modules. We will follow the strategy mentioned in \cite[Remark 2.32]{AIR}. Hence, we consider the opposite categories $\mathcal{A}^{\,\rm op}$ and $\mathcal{B}^{\,\rm op}$ to prove some preliminary lemmas.

Given a covering functor $\F: \mathcal{B} \to \mathcal{A}$, there is an associated covering functor $\F^{\,\rm op}: \mathcal{B}^{\,\rm op} \to \mathcal{A}^{\,\rm op}$ given as follows: the functor $\F^{\,\rm op}$ acts on the objects of $\mathcal{B}^{\,\rm op}$ in the same way as $\F$. If $f: x \to y$ is a morphism in $\mathcal{B}$, then $\F^{\,\rm op}(f^{\,\rm op}) = (\F(f))^{\,\rm op}$. We let
$\F_{\lambda}^{\;\op}:\mathcal{B}^{\op}\MOD\rightarrow \mathcal{A}^{\op}\MOD$ be the push-down functor associated with $\F^{\,\rm op}$.

Recall that for any locally bounded $\mathbb{K}$-linear category $\mathcal{R}$, an indecomposable projective $\mathcal{R}$-module at $x\in\mathrm{obj}(\mathcal{R})$ is finite-dimensional. We have duality
\[
     (-)^{*}\coloneqq\Hom_{\mathcal{R}}(-,\mathcal{R}):\mathsf{proj}\,\mathcal{R}\rightarrow\mathsf{proj}\,\mathcal{R}^{\mathrm{op}}.
\]
Let $X\in\mathcal{R}\mod$ with a minimal projective presentation given by
\[
    \begin{tikzcd}
        P_{-1}\arrow[r,"f"]&    P_{0}\arrow[r]&X\arrow[r]&0.
    \end{tikzcd}
\]
The transpose functor $\mathrm{Tr}$ on $X$ is defined by the exact sequence
\[
    \begin{tikzcd}
        P^*_{0}\arrow[r,"f^*"]&    P^*_{-1}\arrow[r]&\mathrm{Tr}\,X\arrow[r]&0.
    \end{tikzcd}
\]

\begin{lem}
\phantomsection\label{comPD&*}
    Let $Q\in\mathsf{proj}\,\mathcal{B}$ be  indecomposable projective.
    Then $\F_{\lambda}^{\;\op}(Q)^{*}=(\F_{\lambda}Q)^{*}$.
\end{lem}
\begin{proof}
    Let $P$ be indecomposable projective in $\mathcal{A}\mod$ with $\mathcal{F}_\lambda Q = P$. There is a vertex $i$ in $Q$ with $P \cong P_i$, the projective module at vertex $i$. We see that there is a vertex $j$ in $\Gamma_0$ with $\mathcal{F}j = i$ such that $Q \cong \widehat{P}_j$, the projective module at vertex $j$. Recall that $P_{i}^{*}=\Hom_{\mathcal{A}}(P_i,\mathcal{A})$ and $(\widehat{P}_j)^{*}=\Hom_{\mathcal{B}}(\widehat{P}_j,\mathcal{B})$. Hence, $(\widehat{P}_j)^{*}$ is the projective at the vertex $j$ in $\mathcal{B}^{\op}$. Therefore, $\F_{\lambda}^{\;\op}(\widehat{P}_j)^{*}$ is the projective at vertex $j$ in $\mathcal{A}^{\op}$, i.e., $P_i^{*}=(\F_{\lambda}\widehat{P}_j)^{*}$.
\end{proof}

\begin{lem}
\label{ComTr&PD}
    Let $M\in \mathcal{A}\mod$. If $X\in \mathcal{B}\mod$ is such that $\F_{\lambda}X=M$, then $\F_{\lambda}^{\;\op}\tr X=\tr\F_{\lambda}X$.
\end{lem}
\begin{proof}
    Let
    \begin{equation}
    \label{PRtildeM}
        \begin{tikzcd}
            Q_{-1}\arrow[r,"f"]&    Q_{0}\arrow[r]&X\arrow[r]&0
        \end{tikzcd}
    \end{equation}
    be the minimal projective presentation of $X$ with $Q_{-1},Q_{0}\in\mathsf{proj}\,\mathcal{B}$. Define $P_0\coloneqq\F_{\lambda}Q_{0}$ and $P_{-1}\coloneqq\F_{\lambda}Q_{-1}$. We get that
    \begin{equation}
    \label{PRM}
        \begin{tikzcd}
            P_{-1}\arrow[r,"\F_{\lambda}f"]& P_{0}\arrow[r]&M\arrow[r]&0
        \end{tikzcd}
    \end{equation}
    is the minimal projective presentation of $M$ with $P_{-1},P_{0}\in\mathsf{proj}\,\mathcal{A}$. Now we apply the functors $\Hom_{\mathcal{B}}(-,\mathcal{B})$ and $\Hom_{\mathcal{A}}(-,\mathcal{A})$ to the exact sequences in (\ref{PRtildeM}) and (\ref{PRM}), respectively, to get
    \begin{equation}
    \label{trtildeM&trM}
        \begin{tikzcd}
            (Q_{0})^{*}\arrow[r,"f^{*}"]&    (Q_{-1})^{*}\arrow[r]&\tr X\arrow[r]&0\\
            P_{0}^{*}\arrow[r,"(\F_{\lambda}f)^{*}"]&P_{-1}^{*}\arrow[r]&\tr M\arrow[r]&0.
        \end{tikzcd}
    \end{equation}
    Apply the push-down functor $\F_{\lambda}^{\;\op}$ to the first exact sequence in (\ref{trtildeM&trM}) and use Lemma \ref{comPD&*} to get
    \[
        \begin{tikzcd}[column sep=huge]
            P_{0}^{*}\arrow[r,"\F_{\lambda}^{\;\op}f^{*}"]&P_{-1}^{*}\arrow[r]&\F_{\lambda}^{\;\op}\tr X\arrow[r]&0
        \end{tikzcd}
    \]
    Since $\F_{\lambda}^{\;\op}f^{*} \cong (\F_{\lambda}f)^{*}$, we get that $\F_{\lambda}^{\;\op}\tr X=\tr M=\tr\F_{\lambda}X$.
\end{proof}

The following proposition shows that if we start with a support $\tau_{\mathcal{A}}$-tilting pair which is in the essential image of the push-down functor $\F_{\lambda}$, then its right mutation is also in the essential image of $\F_{\lambda}$.

\begin{prop}
\phantomsection\label{rightMutInImg}
    Let $(M,P_{M})$ be a basic support $\tau$-tilting $\mathcal{A}$-module with $M\in\mathcal{A}\mod_{1}$. Assume that $\mu^{+}(M,P_{M})=(N,P_{N})$ is a right mutation of $(M,P_{M})$. Then, there exists $Y\in \mathcal{B}\mod$ and $Q_{Y}\in\mathsf{proj}\,\mathcal{B}$ such that $\F_{\lambda}Y=N$ and $\F_{\lambda}Q_{Y}=P_{N}$.
\end{prop}

\begin{proof}
    We will follow \cite[Remark 2.32]{AIR} in order to understand $(N,P_{N})$. We let $M_{\mathrm{pr}}$ be the maximal projective direct summand of $M$ and we let $M_{\mathrm{np}}$ be such that $M \cong M_{\mathrm{np}} \oplus M_{\mathrm{pr}}$. Recall $(M,P_{M})^{\dag}=(\tr M_{\mathrm{np}}\oplus P_{M}^{*},M_{\mathrm{pr}}^{*})$, which is a support $\tau_{\mathcal{A}^{\rm op}}$ rigid pair. Let $(L,P_{L})$ be the left mutation of $(\tr M_{\mathrm{np}}\oplus P_{M}^{*},M_{\mathrm{pr}}^{*})$ such that $(L,P_{L})^{\dag}=(N,P_{N})$. Rewrite $(L,P_{L})=(L_{\mathrm{np}}\oplus L_{\mathrm{pr}},P_{L})$. Since $(L,P_{L})$ is a left mutation of $(M,P_{M})^{\dag}$, we know as a consequence of Proposition \ref{LeftMutInImg} that all of the direct summands of $L$ are in the essential image of the push-down functor $\F_{\lambda}^{\;\op}$. Therefore, assume we have
    \[
        (L_{\mathrm{np}}\oplus L_{\mathrm{pr}},P_{L})=(\F_{\lambda}^{\;\op}\widehat{L}_{\mathrm{np}}\;\oplus\,\F_{\lambda}^{\;\op}\widehat{L}_{\mathrm{pr}}\,,\,\F_{\lambda}^{\;\op}\widehat{P}_{L}).
    \]
    On applying $(-)^{\dag}$, we get
    \begin{align*}
        (N,P_{N})=(L,P_{L})^{\dag}&=\left(\F_{\lambda}^{\;\op}\widehat{L}_{\mathrm{np}}\,\oplus\,\F_{\lambda}^{\;\op}\widehat{L}_{\mathrm{pr}}\,,\,\F_{\lambda}^{\;\op}\widehat{P}_{L}\right)^{\dag}\\&=\left(\tr\F_{\lambda}^{\;\op}\widehat{L}_{\mathrm{np}}\,\oplus\,(\F_{\lambda}^{\;\op}\widehat{P}_{L})^{*}\,,\,(\F_{\lambda}^{\;\op}\widehat{L}_{\mathrm{pr}})^{*}\right).
    \end{align*}
    Now applying Lemma \ref{ComTr&PD} and Lemma \ref{comPD&*} to $\mathcal{A}^{\op}$ and using the fact that $\left(\F_{\lambda}^{\;\op}\right)^{\op}=\F_{\lambda}$, we get that
    \begin{align*}                                                  (N,P_{N})&=\left((\F_{\lambda}^{\;\op})^{\mathrm{op}}\tr\widehat{L}_{\mathrm{np}}\,\oplus\,(\F_{\lambda}^{\;\op})^{\op}(\widehat{P}_{L})^{*}\,,\,(\F_{\lambda}^{\;\op})^{\op}(\widehat{L}_{\mathrm{pr}})^{*}\right)\\
    &=\left(\F_{\lambda}\tr\widehat{L}_{\mathrm{np}}\,\oplus\,\F_{\lambda}(\widehat{P}_{L})^{*}\,,\,\F_{\lambda}(\widehat{L}_{\mathrm{pr}})^{*}\right).
    \end{align*}
    Hence, take $Y\coloneqq\tr\widehat{L}_{\mathrm{np}}\oplus(\widehat{P}_{L})^{*}$ and $Q_{Y}\coloneqq(\widehat{L}_{\mathrm{pr}})^{*}$.
\end{proof}

\begin{rmk}
    In the above proposition, since $(N,P_N)$ is a support $\tau_{\mathcal{A}}$-tilting pair in $\mathcal{A}\mod$, we see that $(\mathsf{add}_{G}(Y),\mathsf{add}_{G}(Q_{Y}))$ is a support $(G,\tau_{\mathcal{B}})$-tilting pair in $\mathcal{B}\mod$.
\end{rmk}

From above discussion, we get that the push-down functor $\mathcal{F}_{\lambda}$ is compatible with both $(G,\tau_{\mathcal{B}})$-mutation of $G$-orbits of support $(G, \tau_\mathcal{B})$-tilting pairs and mutation of indecomposable summands of (support) $\tau_{\mathcal{A}}$-tilting pairs. More precisely,

\begin{thm}
\phantomsection\label{mainresult}
    We have the following commutative diagram
    \begin{equation}
    \label{MutCommDiag}
        \begin{tikzcd}
            s(G,\tau_{\mathcal{B}})\tilt \arrow[r,"{\mathcal{F}_{\lambda}}"] \arrow[d,"{\mu}"] & s\tau_{\mathcal{A}}\tilt \arrow[d,"{\mu}"] \\
            s(G,\tau_{\mathcal{B}})\tilt \arrow[r,"{\mathcal{F}_{\lambda}}"] & s\tau_{\mathcal{A}}\tilt
        \end{tikzcd}
    \end{equation}
\end{thm}

\begin{proof}
    Let $(\mathcal{T}_{1},\mathcal{P}_{1})\in s(G,\tau_{\mathcal{B}})\tilt$ such that $\mu_{\mathcal{O}(\widehat{X})}(\mathcal{T}_1)=\mathcal{T}_2$. So, we have $(\mathcal{T}_{2},\mathcal{P}_{2})\in s(G,\tau_{\mathcal{B}})\tilt$ for some $\mathcal{P}_{2}\in \mathsf{proj}\,\mathcal{B}$. Let $\F_{\lambda}(\mathcal{T}_{1},\mathcal{P}_{1})=(\mathsf{add}(T_1),\mathsf{add}(P_1))$, and $\F_{\lambda}\widehat{X}=X$. Let $\mu_{X}(T_1)=T_2$ with $(T_1,P_1),(T_2,P_2)\in s\tau_{\mathcal{A}}\tilt$. If $\mu_{X}(T_1)=T_2$ is the left mutation (resp., right mutation), then from Corollary \ref{LeftG-Mut} (resp., Proposition \ref{rightMutInImg}) we get that $\F_{\lambda}(\mathcal{T}_{2},\mathcal{P}_{2})=(\mathsf{add}(T_2),\mathsf{add}(P_2))$.
\end{proof}

As a consequence of the above commutative diagram we see that support $(G,\tau_{\mathcal{B}})$-tilting pairs also enjoy a similar property as support $\tau_{\mathcal{A}}$-tilting pairs.

\begin{cor}
    Let $(\mathcal{T}_1,\mathcal{P}_1)\in s(G,\tau_{\mathcal{B}})\tilt$. If $\widehat{X}\in\mathcal{B}\ind$ is such that it satisfies either $\mathcal{T}_1=\mathsf{add}_{G}(\mathcal{T},\widehat{X})$ and $\mathcal{P}_1 = \mathcal{Q}$; or $\mathcal{P}_1=\mathsf{add}_{G}(\mathcal{Q},\widehat{X})$ and $\mathcal{T}_1 = \mathcal{T}$ for some almost complete support $(G,\tau_{\mathcal{B}})$-tilting pair $(\mathcal{T},\mathcal{Q})$ in $\mathcal{B}\mod$, then there is a unique $(\mathcal{T}_2,\mathcal{P}_2)\in s(G,\tau_{\mathcal{B}})\tilt$, up to equivalence, such that $\mu_{\mathcal{O}(\widehat{X})}(\mathcal{T}_1)=\mathcal{T}_2$.
\end{cor}

We can define left and right mutation of elements in $s(G,\tau_{\mathcal{B}})\tilt$ using the commutative diagram above.

\begin{defn}
    Let $(\mathcal{T}_1,\mathcal{P}_1),(\mathcal{T}_2,\mathcal{P}_2)\in s(G,\tau_{\mathcal{B}})\tilt$. If $\F_{\lambda}(\mathcal{T}_2,\mathcal{P}_2)$ is a left mutation (resp., right mutation) of $\F_{\lambda}(\mathcal{T}_1,\mathcal{P}_1)$, then $(\mathcal{T}_2,\mathcal{P}_2)$ is called the \emph{left mutation} (resp., \emph{right mutation}) of $(\mathcal{T}_1,\mathcal{P}_1)$, and we denote it by $\mu^{-}_{\mathcal{O}(\widehat{X})}(\mathcal{T}_1)=\mathcal{T}_2$ (resp., $\mu^{+}_{\mathcal{O}(\widehat{X})}(\mathcal{T}_1)=\mathcal{T}_2$).
\end{defn}

Similar to the quiver in Definition \ref{MutQui}, we can define a quiver that records the information about left and right mutations of $G$-orbits in $s(G,\tau_{\mathcal{B}})\tilt$.

\begin{defn}
    The \emph{mutation quiver} $Q(s(G,\tau_{\mathcal{B}})\tilt)$ of $\mathcal{B}$ is defined as follows:
    \begin{enumerate}
        \item The set of vertices are the elements of $s(G,\tau_{\mathcal{B}})\tilt$.
        \item We draw an arrow from $(\mathcal{T}_1,\mathcal{P}_1)$ to $(\mathcal{T}_2,\mathcal{P}_2)$ if $(\mathcal{T}_2,\mathcal{P}_2)$ is a left mutation of $(\mathcal{T}_1,\mathcal{P}_1)$.
    \end{enumerate}
\end{defn}

\begin{rmk}
    Theorem \ref{mainresult} implies that the connected components of $Q(s(G,\tau_{\mathcal{B}})\tilt)$ can be identified with some connected components of $Q(s\tau_{\mathcal{A}}\tilt)$ via the push-down functor. It is natural to ask whether all components of $Q(s\tau_{\mathcal{A}}\tilt)$ can be obtained in this way. In the next section, we address this question in the affirmative when $G$ is free.
\end{rmk}

\section{Some remarks on \texorpdfstring{$\mathcal{A}\mod_1$}{A-mod1} and \texorpdfstring{$G$}{G}-gradable modules}

We use Theorem \ref{mainresult} to make some observations on the image of the push-down functor. Recall Definition \ref{MutQui} of the mutation quiver $Q(s\tau_{\mathcal{A}}\tilt)$ of support $\tau_\mathcal{A}$-tilting pairs. For a support $\tau_\mathcal{A}$-tilting pair $(T,P)$, we write $(T, P) \in \mathcal{A}\mod_1$ when both $T, P$ are in the essential image of the push-down functor $\mathcal{F}_\lambda$. Clearly, this condition is equivalent to $T$ being in the essential image of the push-down functor.

\begin{prop} 
\label{Connected component for tau-tilting}
    If $(T_1,P_1),(T_2, P_2) \in s\tau_{\mathcal{A}}\tilt$ are in the same connected component of $Q(s\tau_{\mathcal{A}}\tilt)$, then
    \[
        (T_1, P_1) \in \mathcal{A}\mod_1 \;\text{if and only if}\;\,(T_2, P_2)\in \mathcal{A}\mod_1.
    \]
    In particular, if $(T, P) \in s\tau_{\mathcal{A}}\tilt$ is such that it belongs to the connected component of $Q(s\tau_{\mathcal{A}}\tilt)$ with $(\mathcal{A},0)$ or $(0, \mathcal{A})$ as one of its vertices, then $(T, P)\in \mathcal{A}\mod_1$.
\end{prop}

\begin{proof}
    The results follow from the commutative diagram (\ref{MutCommDiag}) and the fact that the projective $\mathcal{A}$-modules are in the image of the push-down functor.
\end{proof}

\begin{cor}
    If $Q(s\tau_{\mathcal{A}}\tilt)$ is connected, then every $\tau_{\mathcal{A}}$-rigid $\mathcal{A}$-module belongs to $\mathcal{A}\mod_1$. In particular, if $\mathcal{A}$ is a hereditary algebra or if $\mathcal{A}$ is a $\tau_{\mathcal{A}}$-tilting finite algebra, then every $\tau_{\mathcal{A}}$-rigid $\mathcal{A}$-module lies in $\mathcal{A}\mod_1$.
\end{cor}

\begin{proof}
    The first statement follows from Proposition~\ref{Connected component for tau-tilting}. If $\mathcal{A}$ is $\tau$-tilting finite, then it is well known that $Q(s\tau_{\mathcal{A}}\tilt)$ is connected. If $\mathcal{A}$ is hereditary, then it follows from \cite{BUAN2006572} that $Q(s\tau_{\mathcal{A}}\tilt)$ is connected.
\end{proof}

We recall the following definition from \cite[Defintion, page 337]{Bongartz-Gabriel}

\begin{defn}
\phantomsection\label{locrepfin}
    A \emph{locally representation-finite category} is a locally bounded $\mathbb{K}$-category $\mathcal{R}$ such that the number of $V\in \mathcal{R}\ind$ with $V(x)\neq 0$ is finite for each $x\in \mathrm{obj}(\mathcal{R})$.
\end{defn}

Similar to the above Definition \ref{locrepfin}, we give the following definition in terms of $(G,\tau_{\mathcal{B}})$-tilting theory in $\mathcal{B}\mod$.

\begin{defn}
\phantomsection\label{loc(G,tau)tiltingfin}
    We call $\mathcal{B}$ to be $\emph{locally $(G,\tau_{\mathcal{B}})$-tilting finite}$ if, up to isomorphism, the number of $(G,\tau_{\mathcal{B}})$-rigid modules $M\in \mathcal{B}\mod$ satisfying $M(x)\neq 0$ is finite for every $x\in \mathrm{obj}(\mathcal{B})$. 
\end{defn}

One of the important results, due to P. Gabriel, in the covering theory of finite dimensional algebras can be summarized as follows:

\begin{unnumthm}[see \cite{GabUnivCov} Lemma 3.3, Theorem 3.6, also see \cite{MARTINEZVILLA1983277}]
\phantomsection\label{Gab'sresult}
    If $\mathcal{A}$ is a finite-dimensional algebra and $\mathcal{B}$ a Galois covering of $\mathcal{A}$, then $\mathcal{B}$ is locally representation-finite if and only if $\mathcal{A}$ is representation finite. 
\end{unnumthm}

Using Theorem \ref{mainresult} we get a similar result to the above theorem in terms of $\tau_{\mathcal{A}}$-tilting theory. More precisely,

\begin{thm}
    The Galois covering $\mathcal{B}$ is locally $(G,\tau_{\mathcal{B}})$-tilting finite if and only if $\mathcal{A}$ is $\tau_{\mathcal{A}}$-tilting finite.
\end{thm}

\begin{proof}
   Let us start with the sufficiency. Assume that $\mathcal{A}$ is not $\tau_{\mathcal{A}}$-tilting finite. Therefore, the connected component of $Q(s\tau_{\mathcal{A}}\tilt)$ containing $(\mathcal{A},0)$ has infinitely many vertices. Hence, by Proposition \ref{Connected component for tau-tilting}, we get infinitely many non-isomorphic $\tau_{\mathcal{A}}$-rigid $\mathcal{A}$-modules $\lbrace V_i\rbrace_{i\in\mathfrak{I}}$ such that $V_i\in\mathcal{A}\mod_{1}$ for each $i\in\mathfrak{I}$. We have lifts $\{\widehat{V}_i\}_{i\in\mathfrak{I}}$ over $\mathcal{B} \mod$ with $\F_{\lambda}\widehat{V}_i \cong V_i$ for each $i\in\mathfrak{I}$. Moreover, these $\widehat{V}_i$ are all $(G,\tau_\mathcal{B})$-rigid. Since $Q$ has finitely many vertices, there exists a vertex $x\in Q_{0}$ such that $V_i(x)\neq0$ for infinitely many $i \in \mathfrak{I}$. Let us take $\widehat{x}$ in $\Gamma_0$ with  $\F\widehat{x}=x$. Now, for infinitely many $i \in \mathfrak{I}$, there is $g_i \in G$, with $\widehat{V}_i^{g_i}(\widehat{x})\neq0$. Hence, $\mathcal{B}$ is not locally $(G,\tau_{\mathcal{B}})$-tilting finite.

    Let us now prove the necessity. Assume that  $\mathcal{A}$ is $\tau_{\mathcal{A}}$-tilting finite. It follows from Theorem \ref{mainresult} that $\mathcal{B} \ind$ has finitely many $G$-orbits of $(G, \tau_\mathcal{B})$-rigid modules. Moreover, each $y \in \Gamma_0$ supports only finitely many modules in a given $G$-orbit. The result follows from this.
\end{proof}

Recall that Question~\ref{isinimg?} asks if an $\mathcal{A}$-module $M$ with an open orbit under the natural $\GL(\dim M,\mathbb{K})$-action in the variety of $(\dim M)$-dimensional $\mathcal{A}$-modules is in $\mathcal{A}\mod_1$. In this section, we answer this question assuming some conditions on the Galois group $G$ or on the algebra $\mathcal{A}$.

Notice that we can define the following ${G}$-grading on the algebra $\mathcal{A}$:
\begin{equation}
\phantomsection\label{Ggrading}
    \mathcal{A}=\bigoplus_{g\in G}\mathcal{A}_{g},\;\text{where}\;\mathcal{A}_{g}\coloneqq\Hom_{\mathcal{B}}(\mathcal{A},\mathcal{A}^{g})
\end{equation}
that we will call the \emph{induced} $G$-grading from the Galois cover $p$.

A useful result in this direction, due to E. L. Green, can be summarized as follows, see \cite[Theorem 3.2]{EdGreen} or \cite[Theorem 4]{GreenVillaYoshino}: 
Consider $\mathcal{A}$ as a $G$-graded algebra with the induced $G$-grading from the Galois cover $p : (\Gamma,J)\rightarrow (Q,I)$. A module $M\in \mathcal{A}\mod$ is $G$-gradable if and only if $M$ is isomorphic to a module in the essential image of $\mathcal{F}_{\lambda}$. The following proposition is an analogue of \cite[Corollary 4.4]{AmiOpper}, and a result by W. Crawley-Boevey.

\begin{prop}
\phantomsection\label{Zimg}
    Let $\mathcal{A}$ be any locally bounded $\mathbb{K}$-linear category. If $G=\mathbb{Z}$, and $M\in\mathcal{A}\mod$ is a rigid $\mathcal{A}$-module, then $M$ is isomorphic to a module in the essential image of $\mathcal{F}_{\lambda}$.
\end{prop}

\begin{proof}
    As described in (\ref{Ggrading}), $\mathcal{A}$ is $G$-graded. Since $M$ is a finite-dimensional $\mathcal{A}$-module, we can choose a finite full subcategory $\mathcal{D}\subset \mathcal{A}$ such that $M$ considered as $\mathcal{D}$-module is also rigid. Since $M$ is a rigid $\mathcal{D}$-module, we know that it has an open orbit under the natural $\GL(\dim M, \mathbb{K})$-action on the variety of $(\dim M)$-dimensional $\mathcal{D}$-modules. The $G=\mathbb{Z}$-grading on $\mathcal{A}$ induces a $\mathbb{Z}$-grading on $\mathcal{D}$. Hence, by using W. Crawley-Boevey's result or \cite[Corollary 4.4]{AmiOpper} we get that $M$ is $\mathbb{Z}$-gradable. Therefore, the result follows from the above result of E. L. Green.
\end{proof}

Next, we extend the above result to the case where $G$ is free. Let $G=G_0$ be a finitely generated free group on a finite set $S = S_0$ of free generators. Therefore, every element of $G$ is a reduced word in $S$, and hence admits a length with respect to this representation. We call this the \emph{$S$-length} of an element. Note that since $S$ is finite, there are only finitely many elements in $G$ with a given $S$-length. Moreover, it follows from Nielsen-Schreier theorem that any subgroup of a free group is again free. 

For $i \ge 1$, we construct a normal subgroup $G_i$ of $G_{i-1}$ with set of free  generators $S_i$ as follows.  We let $S_i'$ be all elements of $S_{i-1}$ except one chosen element $a_i$ of smallest $S$-length. We let $G_i$ be the normal subgroup of $G_{i-1}$ generated by $S_i'$. We note that $S_i'$ generates $G_i$ as a normal subgroup, but is not, in general, a set of free generators for $G_i$. We have the following.

\begin{lem}
  A set $S_i$ of free generators for $G_i$ is given by $S_i = \{a_i^jxa_i^{-j} \mid x \in S_{i-1}\setminus\{a_i\}, j \ge 0\}$.
\end{lem}

\begin{proof}
    First, by normality of $G_i$ in $G_{i-1}$, it is clear that $S_i \subset G_i$. We also note that $S_{i-1}\setminus\{a_i\} \subset S_{i}$. To prove that $S_i$ is a set of generators for $G_i$, it suffices to prove that conjugating any $S_i$-word by a power of $a_i$ gives again an $S_i$-word. This follows from the fact that \[a_i^j\left(\prod_{k=1}^r a_i^{r_k}x_ka_i^{-r_k}\right)a_i^{-j} = \prod_{k=1}^r a_i^{r_k+j}x_ka_i^{-r_k-j}.\]
    Now, it remains to show that the set $S_i$ is a free generating set. Assume otherwise. Then there is a non-trivial $S_i$-word
    $\prod_{k=1}^s a_i^{r_k}x_ka_i^{-r_k}$ that is equal to $1$. Now, each $x_k$ is a word in $S_{i-1} \setminus \{a_i\}$. Hence, we can see the above expression as a word in $S_{i-1}$ in $G_{i-1}$. We therefore need $-r_k = r_{k+1}$ for $k=1,\ldots, s-1$ together with $r_1 = r_s = 0$. This forces the expression to be a non-trivial word in $S_{i-1}$, hence it cannot be $1$.
\end{proof}

\begin{lem}
\phantomsection\label{aboutG_i}
    Let $G$ be a free group on a finite set $S$ and let $G_i$ be as above. Then
    \begin{enumerate}
        \item For each $i$, we have that $G_{i-1}/G_{i}$ is infinite cyclic.
        \item For each positive integer $r$, there is a positive integer $m_r$ such that the free generating set $S_{m_r}$ of $G_{m_r}$ does not contain any element of $S$-length smaller than $r$.
    \end{enumerate}
\end{lem}

Let $p: (\Gamma,J) \to (Q,I)$ be a Galois $G$-covering where $G$ is finitely generated free. We choose a connected fundamental domain $F_0$ of $\mathcal{A}$ by $G$ in $\Gamma$. Let us fix a vertex $x$ of $Q$. We choose a lift $\widehat{x}$ of $x$ in $\Gamma$ with the property that $\widehat{x}\in F_0$. For every $i\geq 1$, we define $\mathcal{A}_i$ to be the orbit category of $\mathcal{B}$ by $G_i$. We let $(\Omega_i,I_i)$ be a bound quiver for $\mathcal{A}_i$, so that $\mathcal{A}_{i}\cong \mathbb{K}\Omega_{i}/I_{i}$. Hence, we get a Galois $G_i$-covering $p_i: (\Gamma,J) \to (\Omega_i,I_i)$, and we denote the associated covering functor and the push-down functor by $\prescript{}{i}{\mathcal{F}}$ and $\prescript{}{i}{\mathcal{F}}_\lambda$, respectively. Moreover, for each $i\geq 1$, we also have a Galois $G/G_{i}$-covering $p^{i}: (\Omega_i,I_i) \to (Q,I)$, and we denote the associated covering functor and the push-down functor by $\F^{\,i}$ and $\F_{\lambda}^{\,i}$, respectively.

Now, we construct a fundamental domain $F_i$ of $\mathcal{A}_i$ in $\Gamma$ as follows. Let us start with $F_1$. We consider the infinite cyclic group $G/G_1$ and we note that a complete set of coset representatives of $G/G_1$ is given by the powers of $a_1$. Hence, we see that a fundamental domain of $\mathcal{A}_1$ in $\Gamma$ is given by $\{a_1^jv \mid j \in \mathbb{Z}, v \in F_0\}$. Now, suppose that $i \ge 1$ and that $F_i$ has been defined. We note that the powers of $a_{i+1}$ forms a complete set of representatives of the cosets of $G_{i}/G_{i+1}$. Hence, we see that a fundamental domain of $\mathcal{A}_{i+1}$ in $\Gamma$ is given by $\{a_{i+1}^jv \mid j \in \mathbb{Z}, v \in F_i\}$. By construction, we see that we have a chain
$$F_0 \subset F_1 \subset F_2 \subset \cdots$$
and all $F_i$ are connected. We can write
$$F_i = \{a_i^{r_i} \cdots a_2^{r_2}a_1^{r_1} v \mid v \in F_0, r_1, \ldots, r_i \in \mathbb{Z}\}.$$

\begin{lem} \label{Lemma: pushing a_i}
Let $g \in G$. Then there is a non-negative integer $s$ and integers $r_1, r_2, \ldots, r_s$ such that
\[g = a_s^{r_s} \cdots a_2^{r_2}a_1^{r_1}.\]
\end{lem}

\begin{proof}
    Note that there is $g_1 \in G_1$ such that $g = g_1 a_1^{r_1}$. By induction, for $i \ge 1$, there is $g_i \in G_i$ and integers $r_1, r_2, \ldots, r_i$ such that $g = g_i a_i^{r_i} \cdots a_2^{r_2}a_1^{r_1}.$ Let $t_i$ be the $S_i$-length of $g_i \in G_i$. We note that if no power of $a_{i+1}$ appears in the $S_i$-expression for $g_i$, then $t_{i+1}=t_i$. Otherwise, we can write
    $$g_i = x a_{i+1}^py$$
    where $p \ne 0$, $y$ has no power of $a_{i+1}$ in it and $x$ does not end with a non-zero power of $a_{i+1}$. Let $y = y_1y_2 \cdots y_q$ be written as a word in $S_i$. Then
    $$a_{i+1}^p y = \left(\prod_{k=1}^qa_{i+1}^py_k a_{i+1}^{-p}\right)a_{i+1}^{p}$$
    where the product is a word in $S_{i+1}$. Repeating this for all powers of $a_{i+1}$ appearing in the $S_i$ expression for $g_i$, we get  $$g = g_{i+1}a_{i+1}^u a_i^{r_i} \cdots a_2^{r_2}a_1^{r_1}$$
    for some integer $u$ and $t_{i+1} < t_i$. The result now follows.
\end{proof}

\begin{lem}
    We have $\bigcup_{k \ge 0}F_i = \Gamma$.
\end{lem}

\begin{proof}
    Let $y \in \Gamma_0$. Then there is $g \in G$ and $v \in F_0$ such that $gv = y$. Now, it follows from Lemma \ref{Lemma: pushing a_i} that there exists $s$ and integers $r_1, r_2, \ldots, r_s$ such that $g = a_s^{r_s} \cdots a_2^{r_2}a_1^{r_1}$. Hence, $y \in F_s$.
\end{proof}

\begin{lem}\label{lem: encompassing endpoints}
    Let $r$ be a positive integer. Using the above notation, let $i$ be such that the fundamental domain $F_i$ of $\mathcal{A}_i$ contains all the endpoints of walks in $\Gamma$ of length at most $r$ which start at $\widehat{x}$. Suppose $\prescript{}{i}{\mathcal{F}}\,\widehat{x}=\bar x$. If $M\in\mathcal{A}_{i}\mod$ is supported at $\bar x$ and of dimension at most $r$, then there exists $N\in\mathcal{B}\mod$ such that $\prescript{}{i}{\mathcal{F}}_\lambda N \cong M$. 
\end{lem}

\begin{proof}
    Let $z\in\Gamma_{0}\cup\Gamma_{1}$. We define a module $N\in\mathcal{B}\mod$ in the following way: if $z\in F_{i}$, then $N(z)\coloneqq M(\prescript{}{i}{\mathcal{F}}z)$; otherwise $N(z)\coloneqq 0$. Since $M\in\mathcal{A}_{i}\mod$ is of dimension at most $r$ supported at $\bar x$, and $F_i$ contains all the endpoints of walks in $\Gamma$ of length at most $r$ which start at $\widehat{x}$, it follows from the definition of the push-down functor $\prescript{}{i}{\mathcal{F}}_\lambda$ that $\prescript{}{i}{\mathcal{F}}_\lambda N \cong M$.
\end{proof}

\begin{thm}
    Let $p: (\Gamma,J) \to (Q,I)$ be a Galois $G$-covering where $G$ is finitely generated free. If $M\in\mathcal{A}\mod$ is a rigid $\mathcal{A}$-module, then there exists a $N\in\mathcal{B}\mod$ such that $\mathcal{F}_\lambda N \cong M$.
\end{thm}

\begin{proof}
    Using the above notations, we show that, for any $i \ge 1$, the module $M$ lies in the essential image of the push-down functor $\mathcal{F}^{\,i}_\lambda$. We proceed by induction on $i$.
    
    For $i = 1$, this follows from Proposition \ref{Zimg} using that the Galois group $G/G_1$ is infinite cyclic. We let $M_1$ in $\mathcal{A}_1\mod$ be such that $\mathcal{F}^{\,1}_\lambda M_1 \cong M$. We note that $M_1$ has the same (total) dimension as $M$. Using exactness of $\mathcal{F}^{\,1}_\lambda$ and Corollary \ref{non-radtonon-rad}, we see that the image under $\mathcal{F}^{\,1}_\lambda$ of a non-split self-extension is non-split. This shows that $M_1$ is rigid. By induction, we assume that we have constructed a finite dimensional rigid module $M_i$ over $\mathcal{A}_i$ such that $\mathcal{F}_\lambda^{\,i} M_i \cong M$. Again, we note that the dimension of $M_i$ is the same as the dimension of $M$. We consider the Galois $G_{i}/G_{i+1}$-covering $q_i:\mathcal{A}_{i+1} \to \mathcal{A}_i$, where $G_{i}/G_{i+1}$ is infinite cyclic. We let $\prescript{}{i}{\mathcal{F}}^{\,i+1}_\lambda: \mathcal{A}_{i+1} \mod \to \mathcal{A}_i \mod$ be the corresponding push-down functor. Again, by Proposition \ref{Zimg}, it follows that there exists a finite dimensional module $M_{i+1}$ over $\mathcal{A}_{i+1}$ such that $\prescript{}{i}{\mathcal{F}}^{\,i+1}_\lambda M_{i+1} \cong M_i$. Hence, $\mathcal{F}^{\,i+1}_\lambda M_{i+1} =  \mathcal{F}_\lambda^{\,i}\, \prescript{}{i}{\mathcal{F}}^{\,i+1}_\lambda M_{i+1} \cong \mathcal{F}_\lambda^{\,i} M_i  \cong M$.

    Now, assume that $M$ has dimension $r$ and is supported at vertex $x$ in $\mathcal{A}$. Let $\widehat{x}$ be a lift of $x$ in $\Gamma$. Using Lemmas~\ref{lem: encompassing endpoints}, we let $i_r$ be a positive integer such that the fundamental domain $F_{i_r}$ contains all endpoints of walks of length at most $r$ starting at $\widehat{x}$ in $\Gamma$. We let $\prescript{}{i_r}{\mathcal{F}}\,\widehat{x}=\bar x$. The module $M_{i_r}$ constructed above is supported at $\bar x$ and has dimension $r$. By Lemma~\ref{lem: encompassing endpoints}, we conclude there exists a $N\in\mathcal{B}\mod$ such that $\prescript{}{i_r}{\mathcal{F}}_\lambda N \cong M_{i_r}$. Therefore, we have $\F_{\lambda}N=\mathcal{F}_\lambda^{\,i_r} \prescript{}{i_r}{\mathcal{F}}_\lambda N\cong\mathcal{F}_\lambda^{\,i_r}M_{i_r}\cong M$.
\end{proof}

\begin{cor}
    Let $p: (\Gamma,J) \to (Q,I)$ be a Galois $G$-covering where $G$ is finitely generated free. If $M\in\mathcal{A}\mod$ is a rigid $\mathcal{A}$-module, then $M$ is $G$-gradable.
\end{cor}
 
\begin{cor}
    If $\mathcal{A}$ is a monomial algebra and $M$ is a rigid $\mathcal{A}$-module, then $M$ lies in the essential image of any push-down functor.
\end{cor}

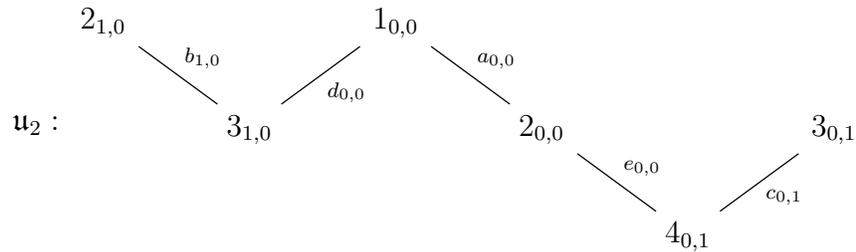
\begin{proof}
    We consider the bound quiver $(Q,I)$ for $\mathcal{A}$ where $I$ is monomial, and we let $\pi: (\widetilde{Q},\widetilde{I}) \to (Q,I)$ be the universal covering. Fix a vertex $x_0\in Q_0$. Since $I$ is monomial, the Galois group is $\pi_1(Q,I) = \pi_1(Q,x_0)$ and is finitely generated free. Hence, $M$ is in the essential image of the corresponding push down functor $\widetilde{\mathcal{F}}_\lambda$. By the universal property of the universal covering, the result follows.
\end{proof}

\begin{rmk}
    If $\mathcal{A}$ is a string algebra, then it follows from the definition of the push-down functor that any string module lies in $\mathcal{A}\mod_1$, and hence is $G$-gradable. However, in general, not every string module is rigid.
\end{rmk}

\section{Example}

\begin{ex}
    Consider the algebra $\mathcal{A}\coloneqq\mathbb{K}Q/I$ given by the quiver
    \[
        \begin{tikzcd}
	       {Q:} & 1 & 2 & 3 & 4
	       \arrow["a", from=1-2, to=1-3]
	       \arrow["d", bend left, from=1-2, to=1-4]
	       \arrow["b", from=1-3, to=1-4]
	       \arrow["e"', bend right, from=1-3, to=1-5]
	       \arrow["c", from=1-4, to=1-5]
        \end{tikzcd}
        \;\;\;\;\text{and}\;\;\;\; I=\langle ba,cb\rangle.
    \]
    Let $M(\mathfrak{u})\in \mathcal{A}\mod$ be the string module given by the following string $\mathfrak{u}$
    
    \[
        \mathfrak{u}:
        \begin{tikzcd}
	       2 && 1 \\
	       & 3 && 2 && 3 \\
	       &&&& 4
	       \arrow["b", no head, from=1-1, to=2-2]
	       \arrow["a", no head, from=1-3, to=2-4]
	       \arrow["d", no head, from=2-2, to=1-3]
	       \arrow["e", no head, from=2-4, to=3-5]
	       \arrow["c", no head, from=3-5, to=2-6]
        \end{tikzcd}
    \]

We illustrate the above constructions on the rigid module $M(\mathfrak{u})$, although it should already be clear that  $M(\mathfrak{u})$ lies in the image of any push-down functor, because it is a string module. Consider the universal $G$-covering $\pi: (\Gamma, J) \to (Q,I)$ where $G = \pi_1(Q,2)=\pi_1(Q, I)$ is the free group on two generators $u,v$ where $u=ad^{-1}b$ and $v=e^{-1}cb$. The module $M(\mathfrak{u})$ is supported on the simple cycle given by $v$. Therefore, we select $a_1 = v$ and we get $G_1 = \langle v^iuv^{-i} \mid i \in Z\rangle$. This gives us a Galois $\mathbb{Z}$-covering $(\Omega_1, I_1) \to (Q,I)$ where $(\Omega_1, I_1)$ is given by 

\tikzcdset{scale cd/.style={every label/.append style={scale=#1},
    cells={nodes={scale=#1}}}}

    \[
        \begin{tikzcd}[scale cd=0.83]
	       &&&&&& \ddots \\
	       &&&& {1_{-1}} & {2_{-1}} & {3_{-1}} & {4_{-1}} \\
	       {(\Omega_{1},I_{1}):} &&& {1_{0}} & {2_{0}} & {3_{0}} & {4_{0}} \\
	       && {1_{1}} & {2_{1}} & {3_{1}} & {4_{1}} \\
	       & {1_{2}} & {2_{2}} & {3_{2}} & {4_{2}} \\
	       &&& \ddots
	       \arrow["{{e_{-2}}}", from=1-7, to=2-8]
	       \arrow["{{a_{-1}}}", from=2-5, to=2-6]
	       \arrow["{{d_{-1}}}", bend left, from=2-5, to=2-7]
	       \arrow["{{b_{-1}}}", from=2-6, to=2-7]
	       \arrow["{{e_{-1}}}", from=2-6, to=3-7]
	       \arrow["{{c_{-1}}}", from=2-7, to=2-8]
	       \arrow["{{a_{0}}}", from=3-4, to=3-5]
	       \arrow["{{d_{0}}}", bend left, from=3-4, to=3-6]
	       \arrow["{{b_{0}}}", from=3-5, to=3-6]
	       \arrow["{{e_{0}}}", from=3-5, to=4-6]
	       \arrow["{{c_{0}}}", from=3-6, to=3-7]
	       \arrow["{{a_{1}}}", from=4-3, to=4-4]
	       \arrow["{{d_{1}}}", bend left, from=4-3, to=4-5]
	       \arrow["{{b_{1}}}", from=4-4, to=4-5]
	       \arrow["{{e_{1}}}", from=4-4, to=5-5]
	       \arrow["{{c_{1}}}", from=4-5, to=4-6]
	       \arrow["{{a_{2}}}", from=5-2, to=5-3]
	       \arrow["{{d_{2}}}", bend left, from=5-2, to=5-4]
	       \arrow["{{b_{2}}}", from=5-3, to=5-4]
	       \arrow["{{e_{2}}}", from=5-3, to=6-4]
	       \arrow["{{c_{2}}}", from=5-4, to=5-5]
        \end{tikzcd}
    \]
    where ${I}_1=\langle b_{i}a_{i},b_{i}c_{i},i\in\mathbb{Z}\rangle$. Let $\mathcal{A}_{1}\coloneqq\mathbb{K}\Omega_1/I_1$. Consider the string module $M(\mathfrak{u}_{1})\in \mathcal{A}_{1}\mod$ given by the following string $\mathfrak{u}_1$ in $\Omega_1$

    \[
        \mathfrak{u}_1:
        \begin{tikzcd}
	       {2_{0}} && {1_0} \\
	       & {3_{0}} && {2_0} && {3_1} \\
	       &&&& {4_1}
	       \arrow["{b_{0}}", no head, from=1-1, to=2-2]
	       \arrow["{a_0}", no head, from=1-3, to=2-4]
	       \arrow["{d_0}", no head, from=2-2, to=1-3]
	       \arrow["{e_0}"', no head, from=2-4, to=3-5]
	       \arrow["{c_1}"', no head, from=3-5, to=2-6]
        \end{tikzcd}
    \]
    Let $\mathcal{F}^{\,1}_{\lambda}$ be the push-down functor associated to the above Galois covering. We see that $\mathcal{F}^{\,1}_{\lambda}\,M(\mathfrak{u}_1)=M(\mathfrak{u})$. Now, we have that $\pi_1(\Omega_1, I_1)$ can be identified with the group $G_1$. Under this identification, we have $u = a_0d_0^{-1}b_0$ and $v=e_{-1}^{-1}c_0b_0$ so that $\pi_1(\Omega_1, I_1) = \langle v^i u v^{-i} \mid i \ge 0 \rangle$. The module $M(\mathfrak{u}_1)$ is supported on the simple cycle given by the free generator $u = v^0 u v^0$. By taking $a_2 = u$, this gives us a Galois $\mathbb{Z}$-covering $(\Omega_2, I_2) \to (\Omega_1,I_1)$ where $(\Omega_2, I_2)$ is given by

    \tikzcdset{scale cd/.style={every label/.append style={scale=#1},
    cells={nodes={scale=#1}}}}
    \[
        \begin{tikzcd}[scale cd=0.6]
	       &&&& \ddots &&& \iddots \\
	       &&& {1_{-1,0}} & {2_{-1,0}} & {3_{-1,0}} & {4_{-1,0}} \\
	       &&&&& \ddots &&&& \iddots \\
	       {(\Omega_{2},I_{2}):} &&&&& {1_{0,-1}} & {2_{0,-1}} & {3_{0,-1}} & {4_{0,-1}} \\
	       && {1_{0,0}} & {2_{0,0}} & {3_{0,0}} & {4_{0,0}} \\
	       &&&& {1_{0,1}} & {2_{0,1}} & {3_{0,1}} & {4_{0,1}} \\
	       &&&&& \iddots & \ddots \\
	       & {1_{1,0}} & {2_{1,0}} & {3_{1,0}} & {4_{1,0}} \\
	       \iddots && {1_{1,1}} & {2_{1,1}} & {3_{1,1}} & {4_{1,1}} \\
	       &&&& \ddots
	       \arrow["{{d_{-2,0}}}", from=1-5, to=2-6]
	       \arrow["{{e_{-1,-1}}}", from=1-8, to=2-7]
	       \arrow["{{a_{-1,0}}}", from=2-4, to=2-5]
            \arrow["{{d_{-1,0}}}", from=2-4, to=5-5]
            \arrow["{{b_{-1,0}}}", from=2-5, to=2-6]
        	\arrow["{{e_{-1,0}}}", from=2-5, to=3-6]
        	\arrow["{{c_{-1,0}}}", from=2-6, to=2-7]
        	\arrow["{{e_{0,-2}}}", from=3-10, to=4-9]
        	\arrow["{{a_{0,-1}}}", from=4-6, to=4-7]
        	\arrow["{{d_{0,-1}}}", bend left, from=4-6, to=4-8]
        	\arrow["{{b_{0,-1}}}", from=4-7, to=4-8]
        	\arrow["{{e_{0,-1}}}", from=4-7, to=5-6]
        	\arrow["{{c_{0,-1}}}", from=4-8, to=4-9]
        	\arrow["{{a_{0,0}}}", from=5-3, to=5-4]
        	\arrow["{{d_{0,0}}}", from=5-3, to=8-4]
        	\arrow["{{b_{0,0}}}", from=5-4, to=5-5]
        	\arrow["{{e_{0,0}}}", from=5-4, to=6-8]
        	\arrow["{{c_{0,0}}}", from=5-5, to=5-6]
        	\arrow["{{a_{0,1}}}", from=6-5, to=6-6]
        	\arrow["{{d_{0,1}}}"', bend left, from=6-5, to=6-7]
        	\arrow["{{b_{0,1}}}", from=6-6, to=6-7]
        	\arrow["{{e_{0,1}}}", from=6-6, to=7-7]
        	\arrow["{{c_{0,1}}}"', from=6-7, to=6-8]
        	\arrow["{{e_{1,-1}}}", from=7-6, to=8-5]
        	\arrow["{{a_{1,0}}}", from=8-2, to=8-3]
        	\arrow["{{d_{1,0}}}", from=8-2, to=9-1]
        	\arrow["{{b_{1,0}}}", from=8-3, to=8-4]
        	\arrow["{{e_{1,0}}}", from=8-3, to=9-6]
        	\arrow["{{c_{1,0}}}", from=8-4, to=8-5]
        	\arrow["{{a_{1,1}}}", from=9-3, to=9-4]
        	\arrow["{{d_{1,1}}}", bend left, from=9-3, to=9-5]
        	\arrow["{{b_{1,1}}}", from=9-4, to=9-5]
        	\arrow["{{e_{1,1}}}", from=9-4, to=10-5]
        	\arrow["{{c_{1,1}}}"', from=9-5, to=9-6]
        \end{tikzcd}
    \]
    where $I_2=\langle b_{i,j}a_{i,j}\,,\,c_{i,j}b_{i,j}\mid i,j\in\mathbb{Z}\rangle$. Let $\mathcal{A}_{2}\coloneqq\mathbb{K}\Omega_2/I_2$. Consider the string module $M(\mathfrak{u}_{2})\in \mathcal{A}_{2}\mod$ given by the following string $\mathfrak{u}_2$ in $\Omega_2$

    \[
        \mathfrak{u}_2:
        \begin{tikzcd}
        	{2_{1,0}} && {1_{0,0}} \\
        	& {3_{1,0}} && {2_{0,0}} && {3_{0,1}} \\
        	&&&& {4_{0,1}}
        	\arrow["{b_{1,0}}", no head, from=1-1, to=2-2]
        	\arrow["{a_{0,0}}", no head, from=1-3, to=2-4]
        	\arrow["{d_{0,0}}"', no head, from=2-2, to=1-3]
        	\arrow["{e_{0,0}}", no head, from=2-4, to=3-5]
        	\arrow["{c_{0,1}}"', no head, from=3-5, to=2-6]
        \end{tikzcd}
    \]
    Let ${\mathcal{F}}^{\,2}_{\lambda}$ be the push-down functor associated to the above Galois covering $(\Omega_2, I_2) \to (Q,I)$. We see that ${\mathcal{F}}^{\,2}_{\lambda}\,{M}(\mathfrak{u}_2)\cong {M}(\mathfrak{u})$, and that ${M}(\mathfrak{u}_2)$ is not supported on any simple cycle in $\Omega_2$. It is clear that
    ${M}(\mathfrak{u}_2)$ lies in the essential image of the push-down functor $_2\mathcal{F}_\lambda: {\mathcal{B}} \mod \to \mathcal{A}_2 \mod$ corresponding to the Galois $G_2$-covering $p_{2}: (\Gamma, J) \to (\Omega_2, I_2)$. Indeed, one just takes the fundamental domain $F_2$ defined above and the ${\mathcal{B}}$-module $M$ with $_2\mathcal{F}_\lambda(M) \cong {M}(\mathfrak{u}_2)$ is simply the module $M$ such that for $y' \in F_2$ with $p_{2}(y')=y$, we set $M(y') = {M}(\mathfrak{u}_2(y))$. We notice that $\mathcal{F}_\lambda(M) \cong M(\mathfrak{u})$.
    
\end{ex}

\section*{Acknowledgement}
This work was supported by National Sciences and Engineering Research Council of Canada [RGPIN-2018-04513 to C.P.,  RGPIN-2024-05714 to D.W.]; and by the Canadian Defence Academy Research Programme.

\end{document}